\documentclass[11pt,oneside,reqno]{amsart}
\usepackage{graphicx}
\usepackage{pict2e}
\usepackage{amssymb}
\usepackage{amsthm}                % better theorem environments
\usepackage[margin=1in]{geometry}  % set the margins to 1in on all sides
\usepackage{graphics}
\usepackage{epstopdf}
\usepackage{amsmath}
\usepackage{graphicx}
\usepackage{subfigure}
\usepackage{latexsym}
\usepackage{amsmath}
\usepackage{amsfonts}
\usepackage{amssymb}
\usepackage{tikz}
\usepackage{mathdots}
\usepackage{comment}
\usepackage{float}
\usepackage[mathscr]{euscript}
\usepackage{enumerate}
\usepackage{epsfig}
\usepackage { graphicx }
\usepackage { hyperref }
\usepackage { graphics }
\usepackage { graphicx }
\usepackage{xparse}
\usepackage{array}

\usepackage {eufrak}
\usepackage[utf8]{inputenc}
\usepackage{longtable}

\usepackage[lite]{amsrefs}
\theoremstyle{definition}
\newtheorem{theorem}{Theorem}[section]
\newtheorem*{theorem*}{Theorem}
\newtheorem{lemma}[theorem]{Lemma}
\newtheorem{corollary}[theorem]{Corollary}
\newtheorem{proposition}[theorem]{Proposition}

\theoremstyle{definition}
\newtheorem{definition}[theorem]{Definition}
\newtheorem{example}[theorem]{Example}
\theoremstyle{remark}
\newtheorem{remark}[theorem]{Remark}
\numberwithin{equation}{section}

\numberwithin{equation}{section}
\begin{document}
%\raggedbottom

\raggedbottom
\title{Twist Regions and  Coefficients Stability of the Colored Jones Polynomial}

\author{Mohamed Elhamdadi}
\address{Department of Mathematics, University of South Florida, 
Tampa, FL 33647 USA}
\email{emohamed@mail.usf.edu }

\author{Mustafa Hajij}
\address{Department of Mathematics, University of South Florida, 
Tampa, FL 33647 USA}
\email{mhajij@usf.edu}

\author{Masahico Saito}
\address{Department of Mathematics, University of South Florida, 
Tampa, FL 33647 USA}
\email{saito@usf.edu }

%%\urladdr{www.math.sc.edu/$\sim$howard} % Delete if not wanted.

%\urladdr{www.math.sc.edu/$\sim$second}

%%%
%%% The following is for the abstract.  The abstract is optional and
%%% if not used just delete, or comment out, the following.
%%%saito@usf.edu

\begin{abstract}
We prove that the coefficients of the colored Jones polynomial of alternating links stabilize under increasing the number of twists in the twist regions of the link diagram. This gives us an infinite family of $q$-power series derived from the colored Jones polynomial parametrized by the color and the twist regions of the alternating link diagram. 
\end{abstract}
 \maketitle

 \tableofcontents
\section{Introduction}
The colored Jones polynomial $J_{n,L}(q)$ of a link $L$ is a sequence of Laurent polynomials with integer coefficients in one variable. The label $n$ stands for the coloring. The polynomial $J_{2,L}(q)$ is the original Jones polynomial. The stability of the highest and lowest coefficients of the colored Jones polynomial has been studied extensively in the last decade. It was conjectured in \cite{DL} that for an alternating link $L$ the lowest $n$ coefficients of $J_{n,L}(q)$ agree with the first $n$ coefficients of $J_{n+1,L}(q)$. This gives a well-defined $q$-series invariant called the tail of the colored Jones polynomial. The term head is used for the highest terms of the colored Jones polynomial. This conjecture was proven by Armond \cite{Armond} and independently by Garoufalidis and L\^e \cite{GL} where higher stability were also shown. This work was extended to quantum spin networks in \cite{Hajij2} and all links in \cite{CL1,CL2}. Dasbach and Lin showed that the head and tail of alternating links contain geometric information that can be used as bounds for the hyperbolic volume of a non-torus alternating links \cite{DL}. This work was extended by Futer, Kalfagianni and Purcell to a larger class of links \cite{Futer1,Futer2}. The tail of the colored Jones polynomial has interesting connections with number theory. It turns out that for many knots with small crossing numbers the tail of the colored Jones polynomial is equal to theta and false theta functions \cite{Cody Oliver, Hajij2}. Moreover, infinite families of classical and new Ramanujan type $q$-series has been recently discovered and recovered using techniques that are related to the tail \cite{CodyOliver,GL,Hajij1,Hajij2,Hajij4}. The tail of the colored Jones polynomial has also been studied using classical $q$-series techniques \cite{GL,Osburn2,KO}. Several connections between twist regions of a knot diagram and the colored Jones polynomial has been made. In \cite{CK1} the authors proved that Mahler measures of the Jones polynomial and the colored Jones polynomial behave like the hyperbolic volume under Dehn surgery and the corresponding full twists. Specifically, they showed that the Mahler measure of the Jones and colored Jones polynomial converges under twisting on any number of strands. A generlized Temperley-Lieb algebra was given in \cite{Cai} to provide an alternative proof of this result. A twisting formula of the Kauffman bracket was used in \cite{CK2} to study the hyperbolic volume of certain families of alternating links. In \cite{CodyOliver} it was proved that the tail of the colored Jones polynomial is invariant under changing the number of twists in any maximal negative twist region.

   In this paper we study the following aspects of the stability of the colored Jones polynomial of alternating links under twists. For each color we study the stability of the coefficients of the colored Jones polynomial as we change the number of crossings in multiple twist regions. The rate of stabiliy is defined and studied as a function of the number of crossings in twists regions and the color. This stability behaviour gives us an infinite family of $q$-power series associated with the colored Jones polynomial of alternating links and parametrized by the color and the twist regions of the link diagrams.  We start by giving the following example to show the behavior we will study. The following computataions were performed in Mathematica package \cite{mathematica} and the formula of the colored Jones polynomial were obtained using the techniques of \cite{MV}.  
 
\begin{example}
\label{main example}
Let $P(c_1,c_2,c_3)$ be the pretzel link presented in Figure \ref{pretzel}.

\begin{figure}[H]
  \centering
   {\includegraphics[scale=0.14]{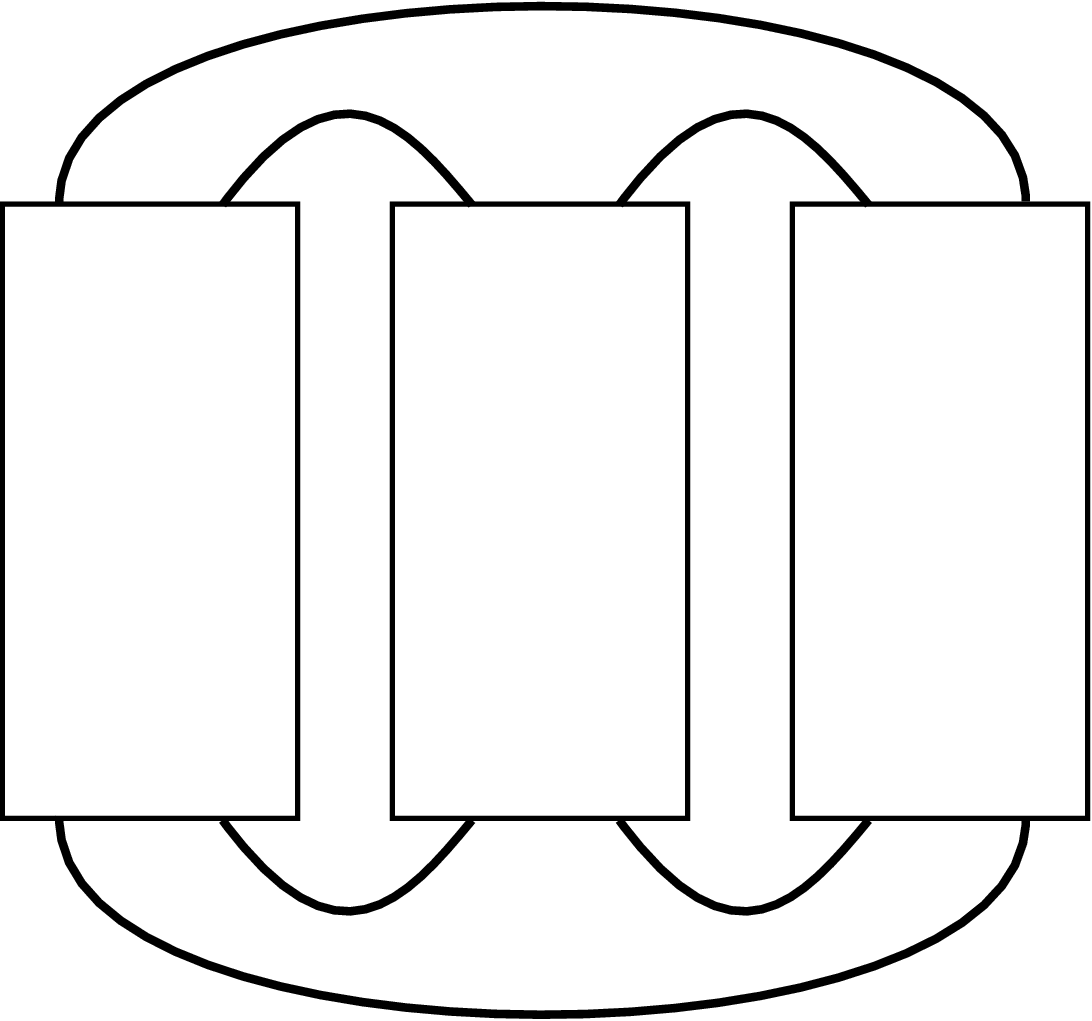}
    \put(-70,33){$c_1$}
          \put(-41,33){$c_2$}
          \put(-15,33){$c_3$}
    \caption{Pretzel link $P(c_1,c_2,c_3)$.}
  \label{pretzel}}
\end{figure} 

\noindent
 Here $c_1$, $c_2$ and $c_3$ are the number of negative crossings in each twist region. See Figure \ref{twist} for an example of a negative twist region with $c$ crossings. 
 The goal of this example is to show the stability behavior of the colored Jones polynomial of the link $P(c_1,c_2,c_3)$ under changing the number of twists in the twists regions of the link diagram $L$.
\begin{figure}[H]
  \centering
   {\includegraphics[scale=.9]{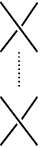}
    \put(-20,30){$c$}
    \caption{A negative twist region with $c$ crossings.}
  \label{twist}}
\end{figure}

\noindent 
In the pretzel link $P(c_1,c_2,c_3)$, let $c_1=8$, $c_2=6$ and let $c_3$ take the values $k \geq 1$.  We calculate the first $k+1$ coefficients of $J_2(P(8,6,k))$ as can be seen in the following table:\\

\begin{small}
 \begin{tabular}{c|>{\raggedright\arraybackslash$}p{6.5cm}<{$}c}
 \hline
 \text{The link $P(8, 6, k)$ } & \text{List of lowest $k+1$ of coefficients of $J_2({P(8, 6, k)})$}\\
  \hline
$k=1$ & $1,-1$  \\
$k=2$ & $1,-1,3$  \\ 
$k=3$ & $1,-1,3,-4$ \\
$k=4$ & $1,-1,3,-4,6$ \\
$k=5$ & $1,-1,3,-4,6,-8$ \\
$k=6$ & $1,-1,3,-4,6,-8,10$ \\
$k=7$ & $1,-1,3,-4,6,-8,10,-11$ \\
$k=8$ & $1,-1,3,-4,6,-8,10,-11,13$ \\
$k=9$ & $1,-1,3,-4,6,-8,10,-11,13,-13$ \\
$k=10$ & $1,-1,3,-4,6,-8,10,-11,13,-13,14$ \\
\hline 
\end{tabular}
\end{small}\\

\noindent
This stability behavior also holds when we change the number of crossings in two crossing regions at the same time as can be seen from the following table.\\

\begin{small}
 \begin{tabular}{c|>{\raggedright\arraybackslash$}p{6.5cm}<{$}c}
 \hline
 \text{The link $P(k, k, 2)$ } & \text{List of lowest $k+1$ of coefficients of $J_2({P(k, k, 2)})$}\\
  \hline
$k=1$ & $1,-1$  \\
$k=2$ & $1,-1,3$  \\ 
$k=3$ & $1,-1,3,-3$ \\
$k=4$ & $1,-1,3,-3,5$ \\
$k=5$ & $1,-1,3,-3,5,-6$ \\
$k=6$ & $1,-1,3,-3,5,-6,7$ \\
$k=7$ & $1,-1,3,-3,5,-6,7,-8$ \\
$k=8$ & $1,-1,3,-3,5,-6,7,-8,9$ \\
$k=9$ & $1,-1,3,-3,5,-6,7,-8,9,-10$ \\
$k=10$ & $1,-1,3,-3,5,-6,7,-8,9,-10,11$ \\
\hline 
\end{tabular}
\end{small}\\

\noindent
Furthermore, the stability behavior occurs also for higher colors. However, for higher colors more coefficients stabilize as we increase the number of crossings. For instance the following table shows a list of coefficients of the third colored Jones polynomial for a sequence of pretzel links.\\ 

\begin{tiny}
 \begin{tabular}{c|>{\raggedright\arraybackslash$}p{6.5cm}<{$}c}
 \hline
 \text{The link $P(k+2, k+4, k+1)$ } & \text{List of lowest $3k+1$ of coefficients of $J_3({P(k+2, k+4, k+1)})$}\\
  \hline
$k=1$ & $1,-1,-1,0$  \\
$k=2$ & $1,-1,-1,0,4,0,-4$  \\ 
$k=3$ & $1,-1,-1,0,4,0,-4,-5,7,6$ \\
$k=4$ & $1,-1,-1,0,4,0,-4,-5,7,6,-1,-13,1$ \\
$k=5$ & $1,-1,-1,0,4,0,-4,-5,7,6,-1,-13,1,7,9,-8$ \\
$k=6$ & $1,-1,-1,0,4,0,-4,-5,7,6,-1,-13,1,7,9,-8,-3,-5,5$ \\
$k=7$ & $1,-1,-1,0,4,0,-4,-5,7,6,-1,-13,1,7,9,-8,-3,-5,5,-1,13,-4$ \\
\hline 
\end{tabular}
\end{tiny}\\

\noindent
Finally, if we increase the number of crossings in twist regions along with the color we also obtain the stability behaviors shown in the following table:\\

\begin{small}
 \begin{tabular}{c|>{\raggedright\arraybackslash$}p{6.5cm}<{$}c}
 \hline
 \text{The link $P(2,5, k)$ } & \text{List of lowest $k+1$ of coefficients of $J_k({P(2, 5, k)})$}\\
  \hline
$k=1$ & $1,-1$  \\
$k=2$ & $1,-1,-1$  \\ 
$k=3$ & $1,-1,-1,0$ \\
$k=4$ & $1,-1,-1,0,0$ \\
$k=5$ & $1,-1,-1,0,0,1$ \\
$k=6$ & $1,-1,-1,0,0,1,0$ \\
$k=7$ & $1,-1,-1,0,0,1,0,1$ \\
\hline 
\end{tabular}
\end{small}\\

\end{example}

\noindent
An explicit formula for the Kauffman bracket of pretzel links is given in \cite{pretzelknots} and a recursive formula can be found in \cite{voy}.

The purpose of this paper is to prove that the stability behavior illustrated in this example holds for the colored Jones polynomial of any sequence of alternating links indexed by the color and the twist regions. 
\subsection{Main Results}
 We start by 
 %proving
  stating the first main result on
  the stability of the coefficients for Kaufmman bracket, or equivalently the Jones polynomial, under increasing the number of twist in a single maximal twist region in:\\

{\bf Theorem } \ref{new greatness}{ \em
Let $L=L_k$ be an alternating link diagram with a marked maximal negative twist region with $k \geq 1$ crossings. Then,
\begin{equation*}
 \langle L_{k} \rangle  \doteq_{4k} \langle L_{k-1} \rangle.
\end{equation*}}\\
Here the notation $P_1\doteq_nP_2$ means that the lowest $ n$ coefficients of the Laurent polynomials $P_1$ and $P_2$ coincide up to a sign.
This theorem generalizes to alternating links with multiple marked twist regions in:\\

{\bf Theorem} \ref{2} { \em  Let $L=L(k_1, \ldots,k_r)$ be an alternating link diagram. Then
\begin{equation*}
\langle L(k_1,\ldots, k_r)\rangle \doteq_{4k} \langle L(k_1-1,\ldots, k_r-1) \rangle,
\end{equation*}
where $k=\min_{1 \leq i \leq r }(k_i)$.
}\\

Theorem \ref{new greatness} generalizes to the unreduced colored Jones polynomial:\\

{\bf Theorem } \ref{queen theorem}{ \em Let $L=L_k$ be a reduced alternating link diagram with a marked maximal negative twist region with $k \geq 1$ crossings. Then,
\begin{equation*}
\tilde{J}_n(L_{k})\doteq_{4n(k-1)+4}\tilde{J}_n(L_{k-1}).
\end{equation*} 
}\\

Moreover, Theorem \ref{2} also generalizes to the unreduced colored Jones polynomial:\\

{\bf Theorem } \ref{nice}{ \em  	 
Let $L=L(k_1,\ldots,k_r)$ be a reduced alternating link diagram. Then
\begin{equation*}
\tilde{J}_n( L(k_1,\ldots, k_r))\rangle \doteq_{n(k-1)+1}  \tilde{J}_n( L(k_1-1,\ldots, k_r-1) 
\end{equation*}
where $k=\min_{1 \leq i \leq r }(k_i)$.
}\\

The previous theorems give us a method to generate families of $q$-series parametrized by the color and the twist regions of the alternating link diagram:\\

 { \bf Theorem} \ref{Cor4.9}{ \em Let $L=L(k_1,\ldots,k_r)$ be an alternating link diagram. Then the sequence $\{\tilde{J}_n(L(k_1+i,\ldots,k_r+i)) \}_{i \in \mathbb{N}}$ has a well-defined tail.
}
\\

Finally we show that our work generalizes the work Armond, Dasbach \cite{CodyOliver} and Garoufalidis, L\^e \cite{GL}:\\

{ \bf Theorem} \ref{cody_us}{ \em Let $L=L(k_1, \ldots, k_r) $ be a reduced alternating diagram. 
Then the sequence 
$\{ \tilde{J}_{n+i} (L( k_1 +i, \ldots, k_r +i) ) \}_{i\in {\mathbb N}} $ 
has a well-defined tail.
}

\subsection{Organization of The Paper}
 The paper is organized as follows. In section \ref{review}, we review the basic material needed from skein theory.  Section \ref{sec3} contains the proof of the result in the case of the Jones polynomial. In section \ref{sec4}, we prove our result for the colored Jones polynomial. In section \ref{sec5} we show the connection of this stability behavior with the tail of the colored Jones polynomial. We conclude the paper by section \ref{sec6} which contains some open questions.
\section{Setting the Scene}
\label{review}
Let $M$ be an oriented $3$-manifold and let $I$ be a closed interval. A framed link in $M$ is an oriented embedding of a disjoint union of
oriented annuli in $M$. If the manifold $M$ has a boundary $\partial M$ then a closed interval in $\partial M$ is called a framed point. A \text{band} in $M$ is an oriented embedding of $I\times I$ that meets the boundary $\partial M$ orthogonally at two framed points. 
\begin{definition}\cite{Przytycki, RT} 
	Let $M$ be an oriented $3$-manifold and $\mathcal{R}$ be a commutative ring with an identity and an invertible element $A$. Let $\mathcal{L}_M$ be the set of isotopy classes of framed links in $M$ including the empty link. If $M$ has a boundary and an even number of marked framed points on $\partial M$ then the set $\mathcal{L}_M$ also includes the isotopy classes of bands that meet the marked points. Let $\mathcal{R}\mathcal{L}_M$ be the free $ \mathcal{R}$-module generated by the set $\mathcal{L}_M$. The \textit{Kauffman Bracket Skein Module} of $M$ and $\mathcal{R}$, denoted by $\mathcal{S}(M,\mathcal{R})$ is the quotient module
	$\mathcal{S}(M,\mathcal{R})=\mathcal{R}\mathcal{L}_M/K(M)$,
	where $K(M)$ is the smallest submodule of $ \mathcal{R}\mathcal{L}_M$ that is generated by all expressions of the form \cite{Kauffman1}:
	\begin{eqnarray*}(1)\hspace{3 mm}
		\begin{minipage}[h]{0.06\linewidth}
			\vspace{0pt}
			\scalebox{0.04}{\includegraphics{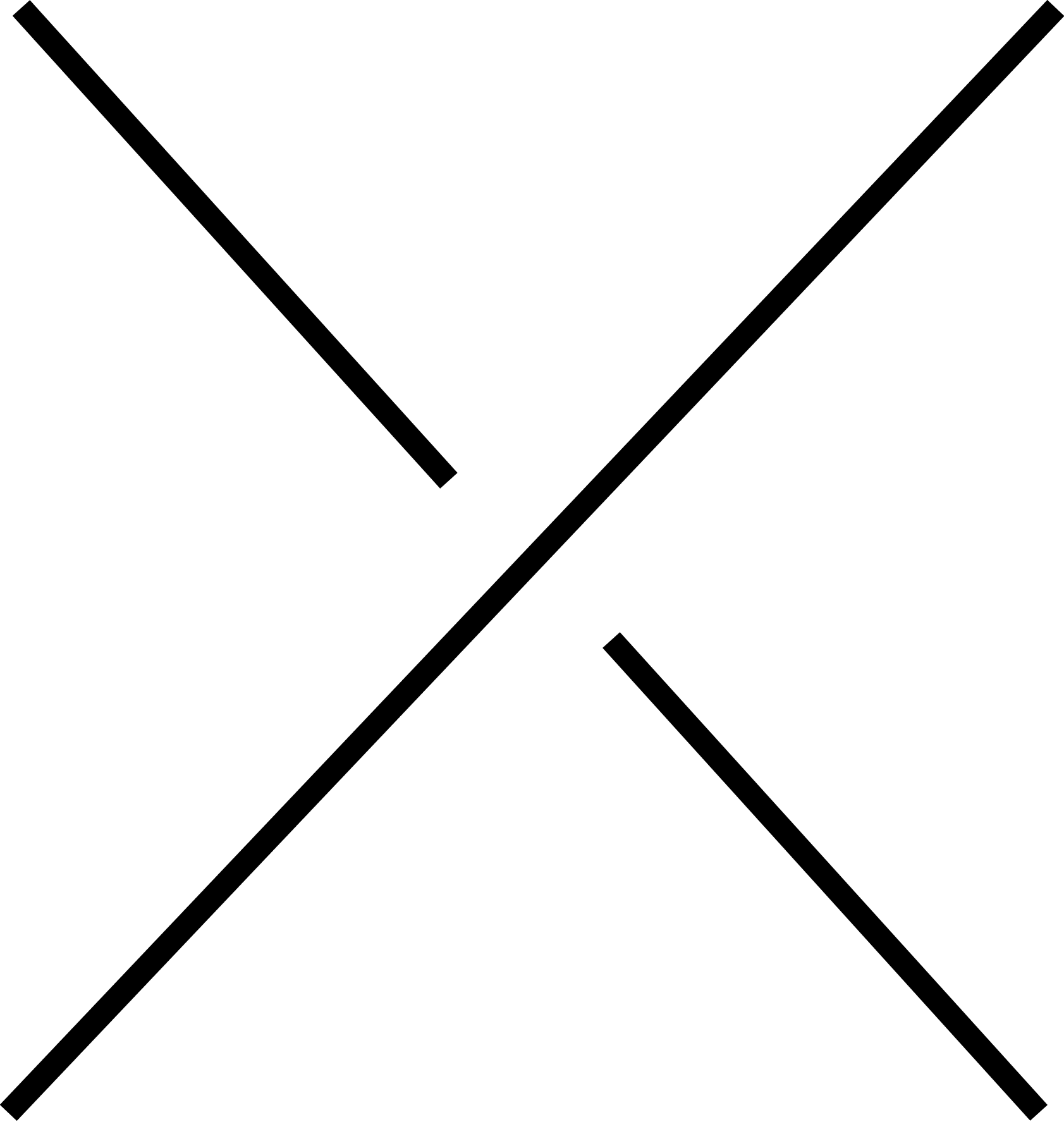}}
		\end{minipage}
		-
		A 
		\begin{minipage}[h]{0.06\linewidth}
			\vspace{0pt}
			\scalebox{0.04}{\includegraphics{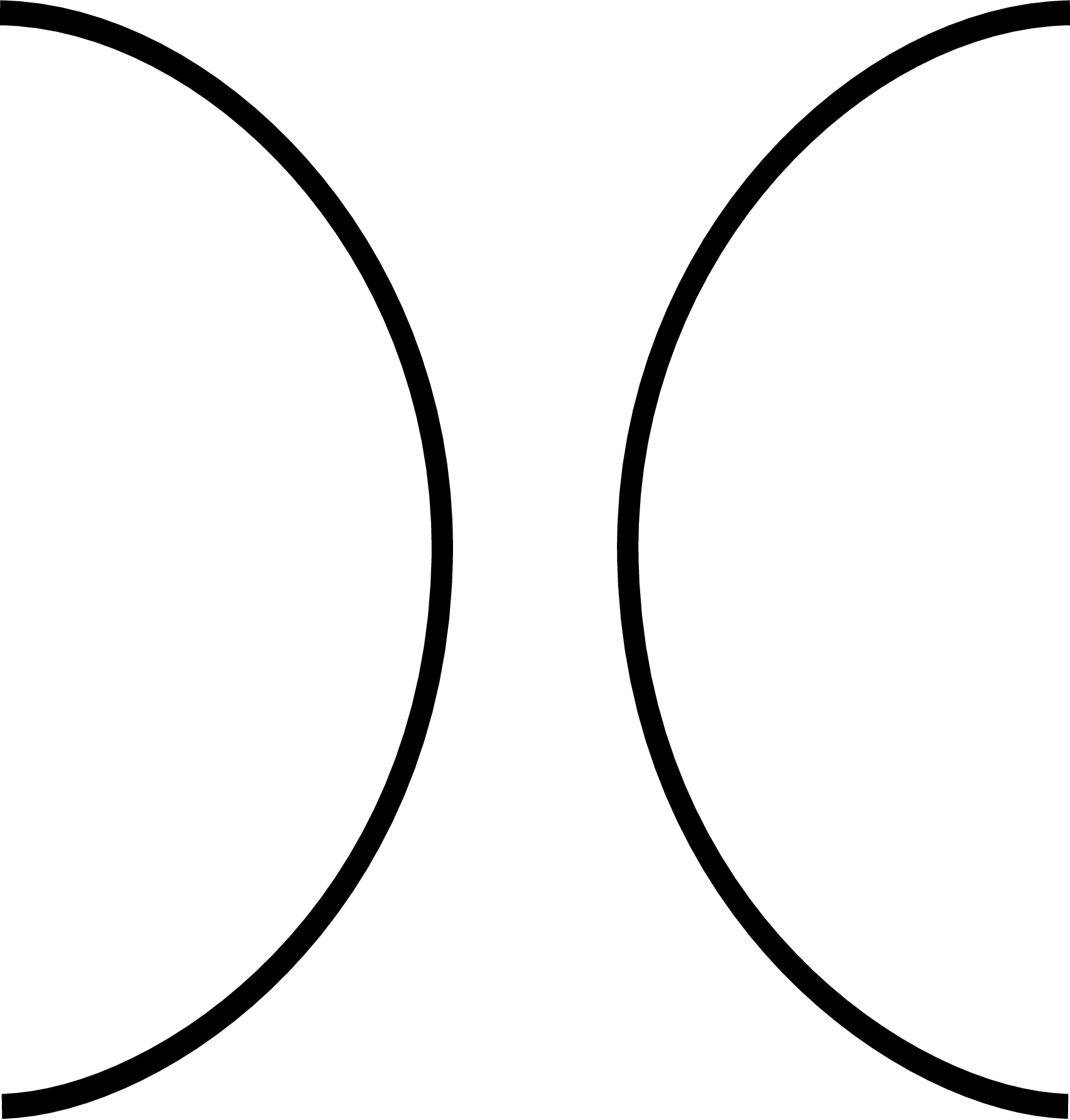}}
		\end{minipage}
		-
		A^{-1} 
		\begin{minipage}[h]{0.06\linewidth}
			\vspace{0pt}
			\scalebox{0.04}{\includegraphics{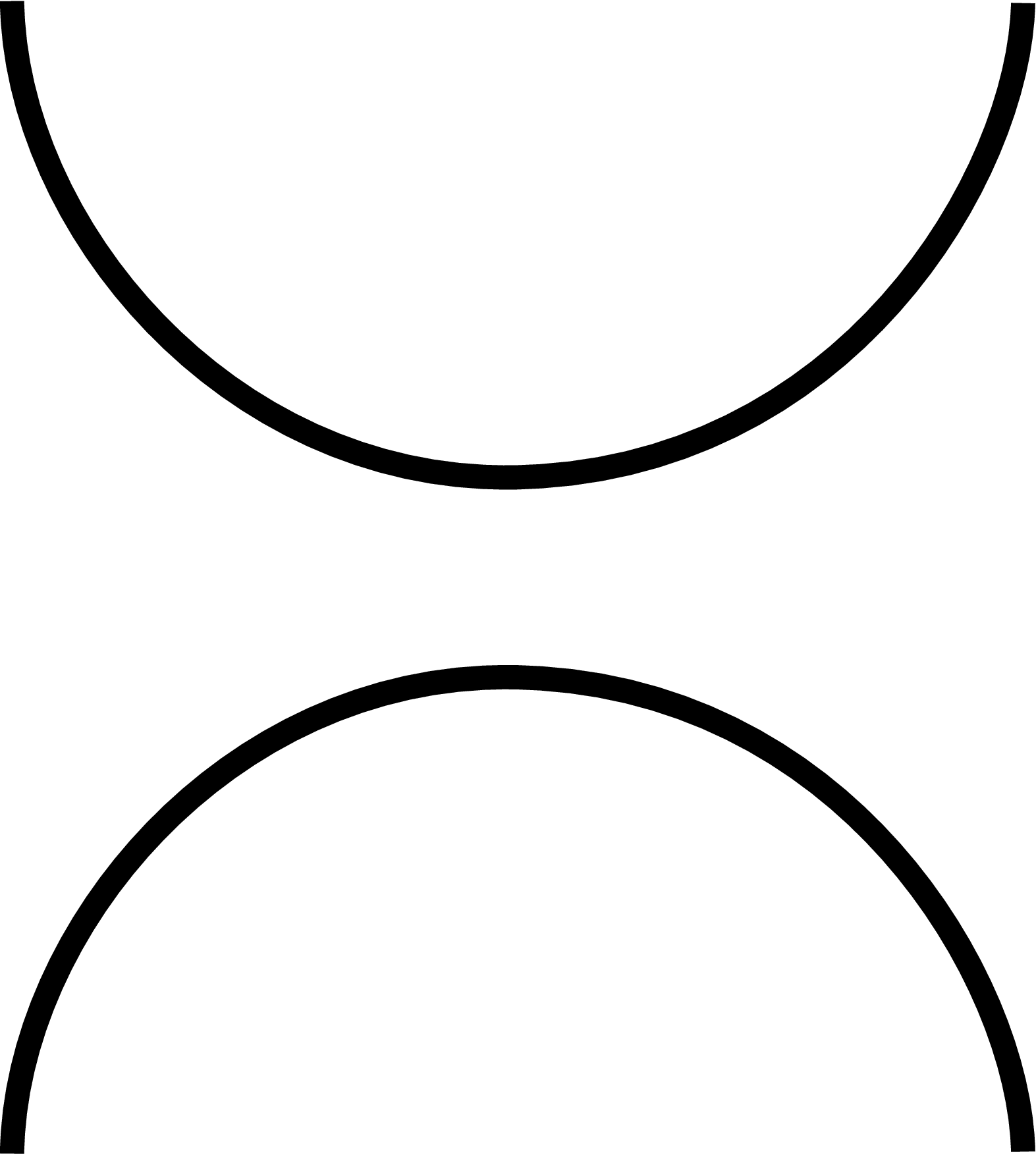}}
		\end{minipage}
		, \hspace{20 mm}
		(2)\hspace{3 mm} L\sqcup
		\begin{minipage}[h]{0.05\linewidth}
			\vspace{0pt}
			\scalebox{0.02}{\includegraphics{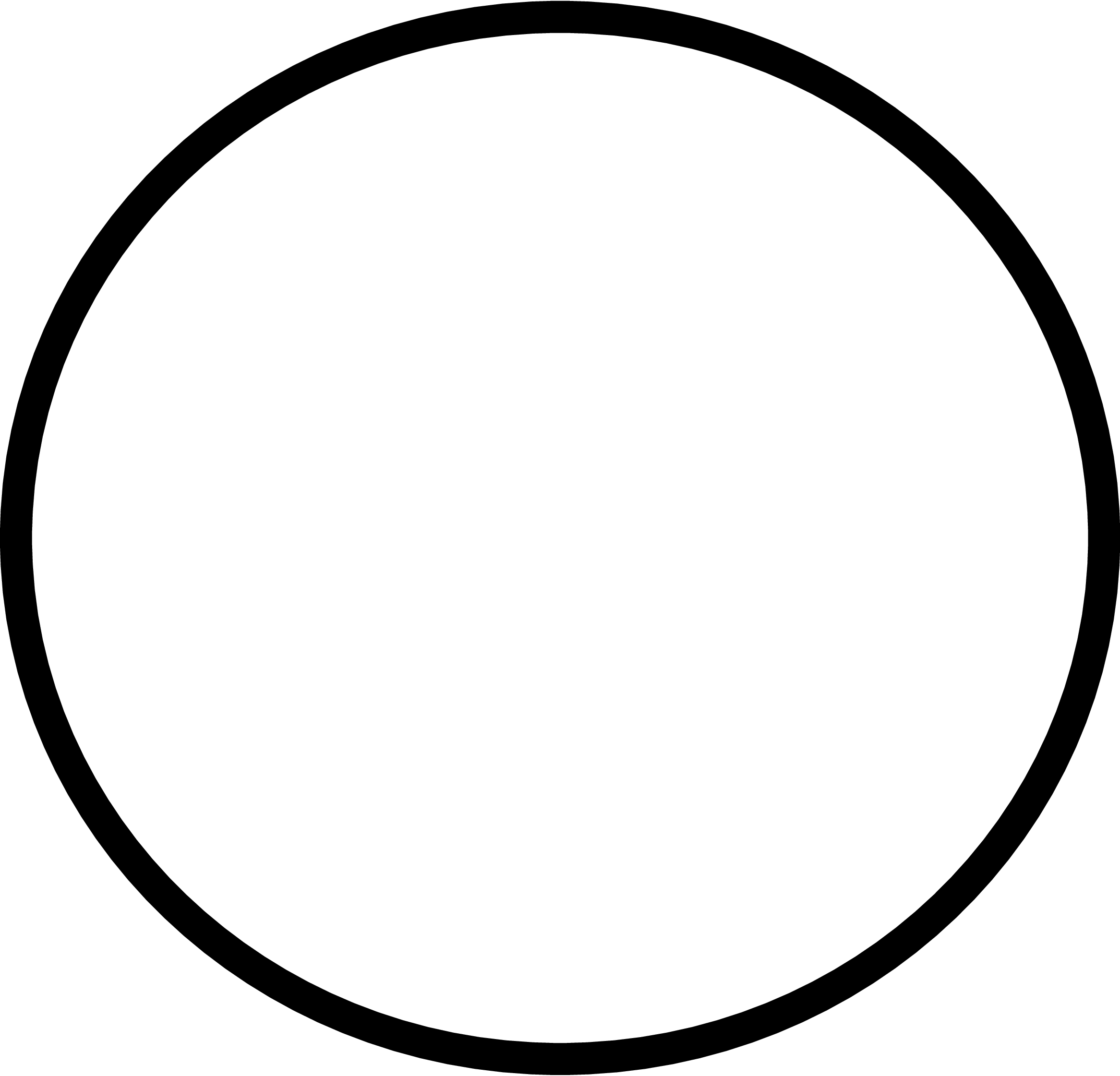}}
		\end{minipage}
		+
		(A^{2}+A^{-2})L. 
	\end{eqnarray*}
	where $L\sqcup$ \begin{minipage}[h]{0.05\linewidth}
		\vspace{0pt}
		\scalebox{0.02}{\includegraphics{simple-circle}}
	\end{minipage}  consists of an element in $\mathcal{L}_M$ and the zero-framed knot 
	\begin{minipage}[h]{0.05\linewidth}
		\vspace{0pt}
		\scalebox{0.02}{\includegraphics{simple-circle}}
	\end{minipage} that bounds a disk in $M$.
\end{definition}    
Throughout this article, the ring $\mathcal{R}$ will be fixed to be the ring of all rational functions $\mathbb{Q}(A)$ with the indeterminate $A$ so we will drop the ring from the notation of the Kauffman bracket skein module. Furthermore, if $F$ is an oriented surface and $M=F \times I$ then we will refer to $\mathcal{S}(M)$ by $\mathcal{S}(F)$ and refer to this module by the Kauffman bracket skein module of $F$.

The Kauffman bracket induces an isomophism $\langle \hspace{2pt}   \rangle : \mathcal{S}(S^2) \longrightarrow \mathbb{Q}(A)$ given by sending every diagram $D$ to $\langle D\rangle$. In particular this isomorphism sends the empty link to $1$. Let $D=I\times I$ be the unit square. Fix $2n$ marked points on the boundary of $D$ such that we have $n$ points on the top of $D$ and $n$ points on the bottom of it. Denote by $\mathcal{S}(D^2,2n)$ the Kauffman bracket skein module of the disk $D$ with the $2n$ marked points. This module can be made into an associative algebra over $\mathbb{Q}(A)$ by the obvious vertical diagram juxtaposition. This algebra, known as the \textit{$n^{th}$ Temperley-Lieb algebra} $TL_n$, plays a central role in Witten-Reshetikhin-Turaev Invariants for $SU(2)$ \cite{Lic92,RT}, the colored Jones polynomial and its applications \cite{RT,TuraevViro92}, and quantum spin networks \cite{MV}. See also the book \cite{Ohtsuki} for more details. For every $n \geq 1 $ there exists a unique idempotent in $TL_n$ known as the $n^{th}$ \textit{Jones-Wenzl idempotent} (projector), denoted $f^{(n)}$. The Jones-Wenzl idempotent was defined by Jones in \cite{J2} and it satisfies the following recursive formula due to Wenzl \cite{Wenzl}: 
\begin{align}
\label{recursive}
  \begin{minipage}[h]{0.05\linewidth}
        \vspace{0pt}
        \scalebox{0.12}{\includegraphics{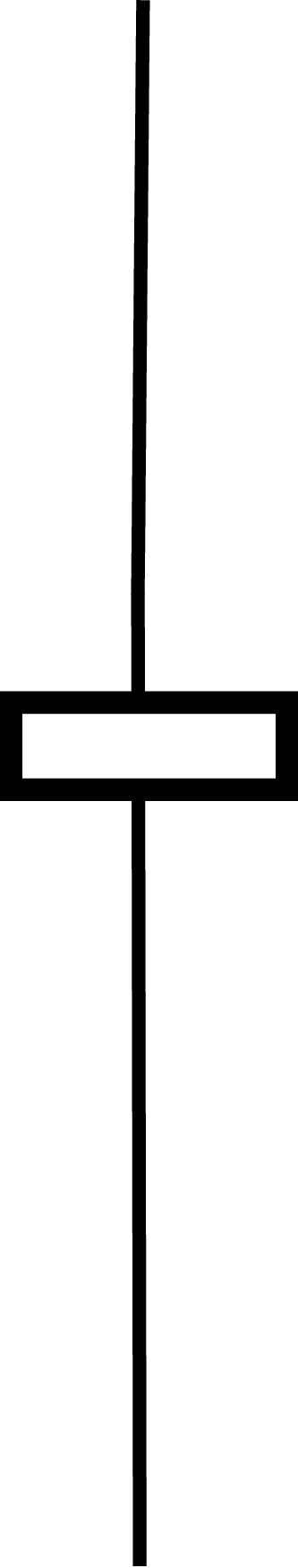}}
         \put(-20,+70){\footnotesize{$n$}}
   \end{minipage}
   =
  \begin{minipage}[h]{0.08\linewidth}
        \hspace{8pt}
        \scalebox{0.12}{\includegraphics{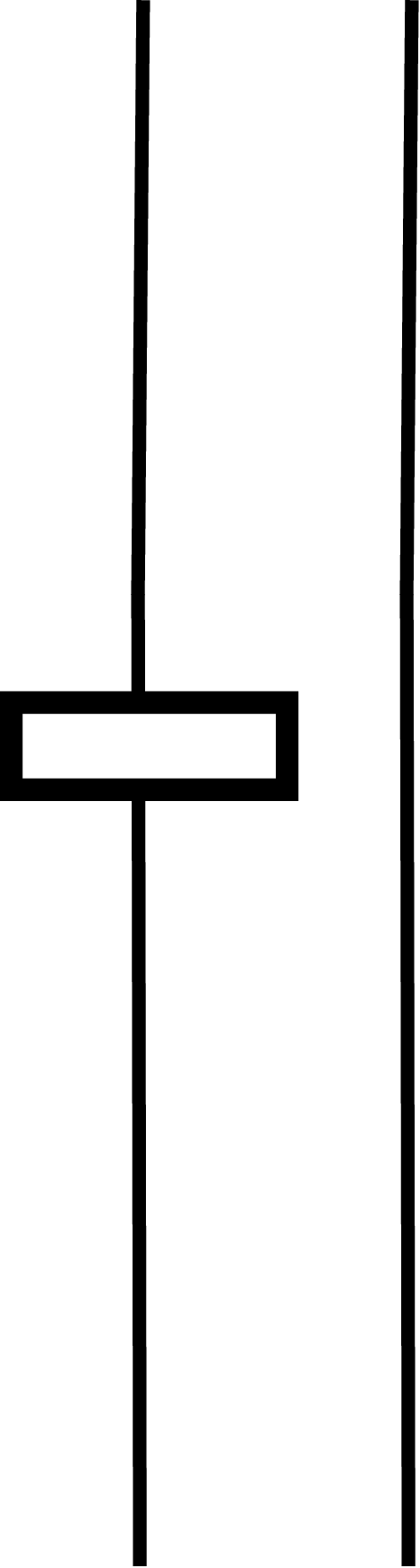}}
        \put(-42,+70){\footnotesize{$n-1$}}
        \put(-8,+70){\footnotesize{$1$}}
   \end{minipage}
   \hspace{9pt}
   -
 \Big( \frac{\Delta_{n-2}}{\Delta_{n-1}}\Big)
  \hspace{9pt}
  \begin{minipage}[h]{0.10\linewidth}
        \vspace{0pt}
        \scalebox{0.12}{\includegraphics{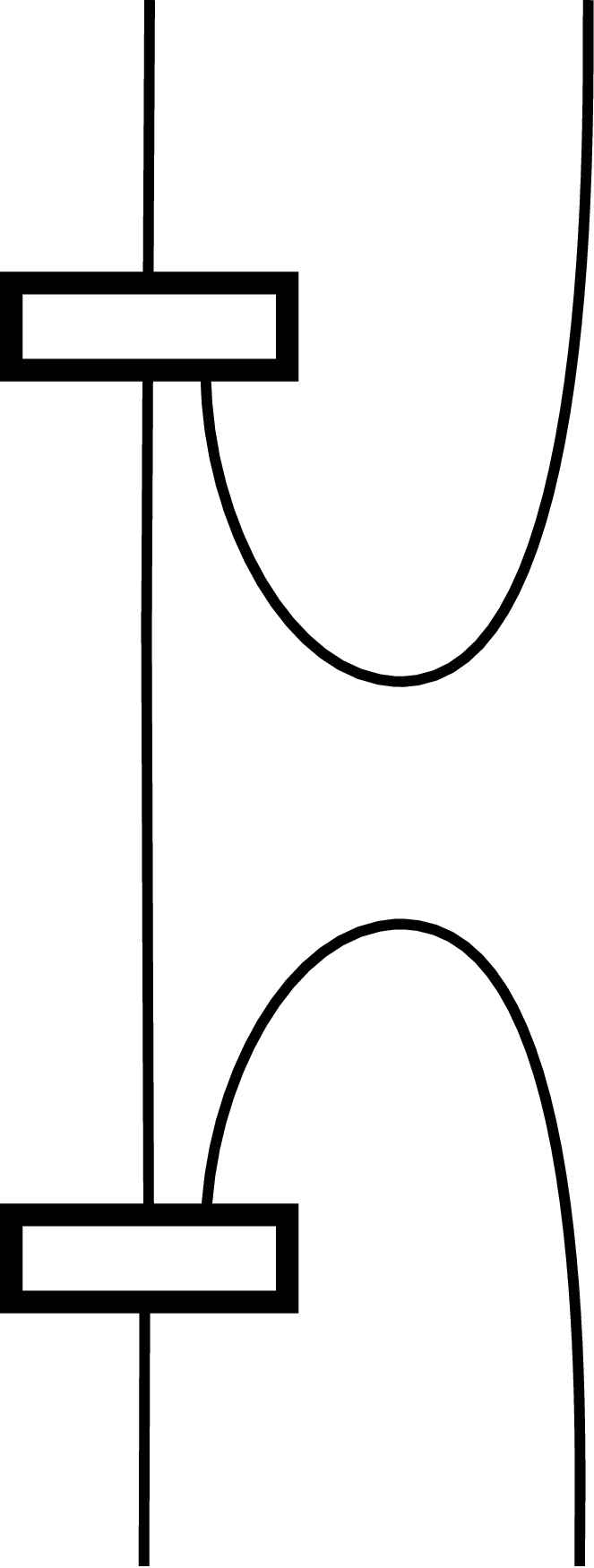}}
         \put(2,+85){\footnotesize{$1$}}
         \put(-52,+87){\footnotesize{$n-1$}}
         \put(-25,+47){\footnotesize{$n-2$}}
         \put(2,+10){\footnotesize{$1$}}
         \put(-52,+5){\footnotesize{$n-1$}}
   \end{minipage}
  , \hspace{20 mm}
    \begin{minipage}[h]{0.05\linewidth}
        \vspace{0pt}
        \scalebox{0.12}{\includegraphics{nth-jones-wenzl-projector}}
        \put(-20,+70){\footnotesize{$1$}}
   \end{minipage}
  =
  \begin{minipage}[h]{0.05\linewidth}
        \vspace{0pt}
        \scalebox{0.12}{\includegraphics{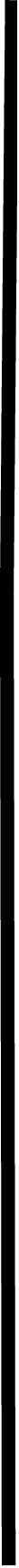}}
   \end{minipage}   
  \end{align}
  where 
\begin{equation*}
 \Delta_{n}=(-1)^{n}\frac{A^{2(n+1)}-A^{-2(n+1)}}{A^{2}-A^{-2}}.
\end{equation*} 

The element $f^{(n)}$ is characterized by the following properties \cite{Lic92}:
\begin{eqnarray}
\label{properties}
\hspace{0 mm}
\begin{minipage}[h]{0.21\linewidth}
\vspace{0pt}
\scalebox{0.115}{\includegraphics{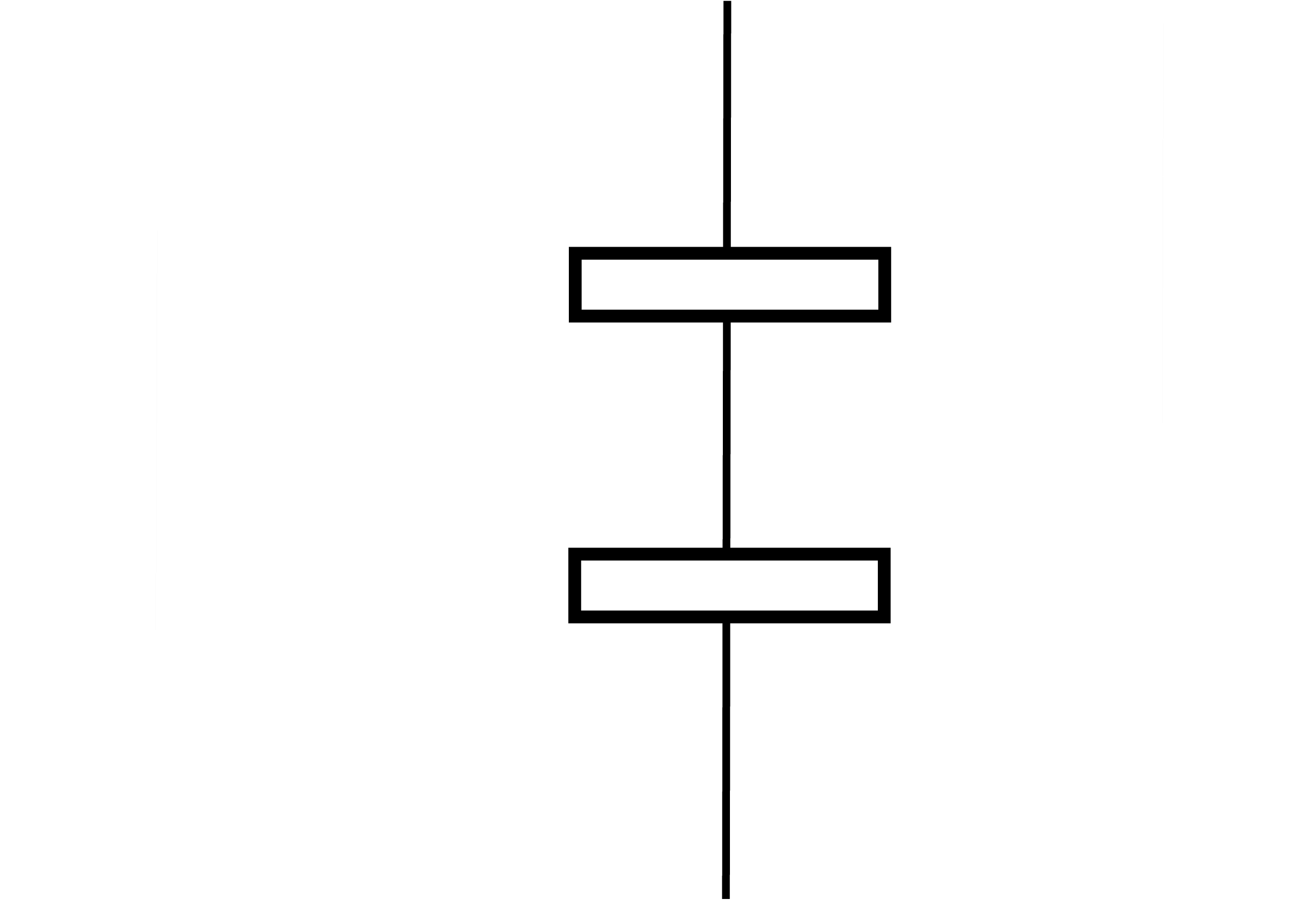}}
\put(0,+80){\footnotesize{$n$}}

\end{minipage}
= \hspace{5pt}
\begin{minipage}[h]{0.1\linewidth}
\vspace{0pt}
\hspace{100pt}
\scalebox{0.115}{\includegraphics{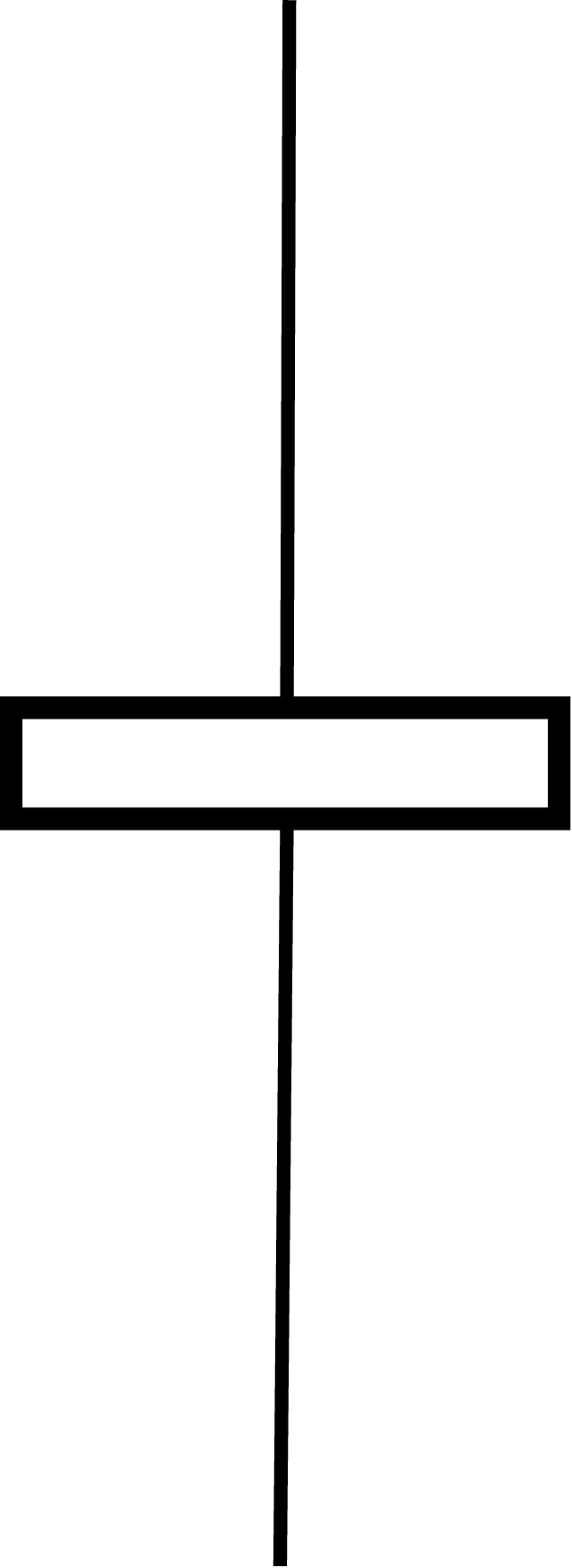}}
\put(-60,80){\footnotesize{$n$}}
\end{minipage}
, \hspace{35 mm}
\begin{minipage}[h]{0.09\linewidth}
\vspace{0pt}
\scalebox{0.115}{\includegraphics{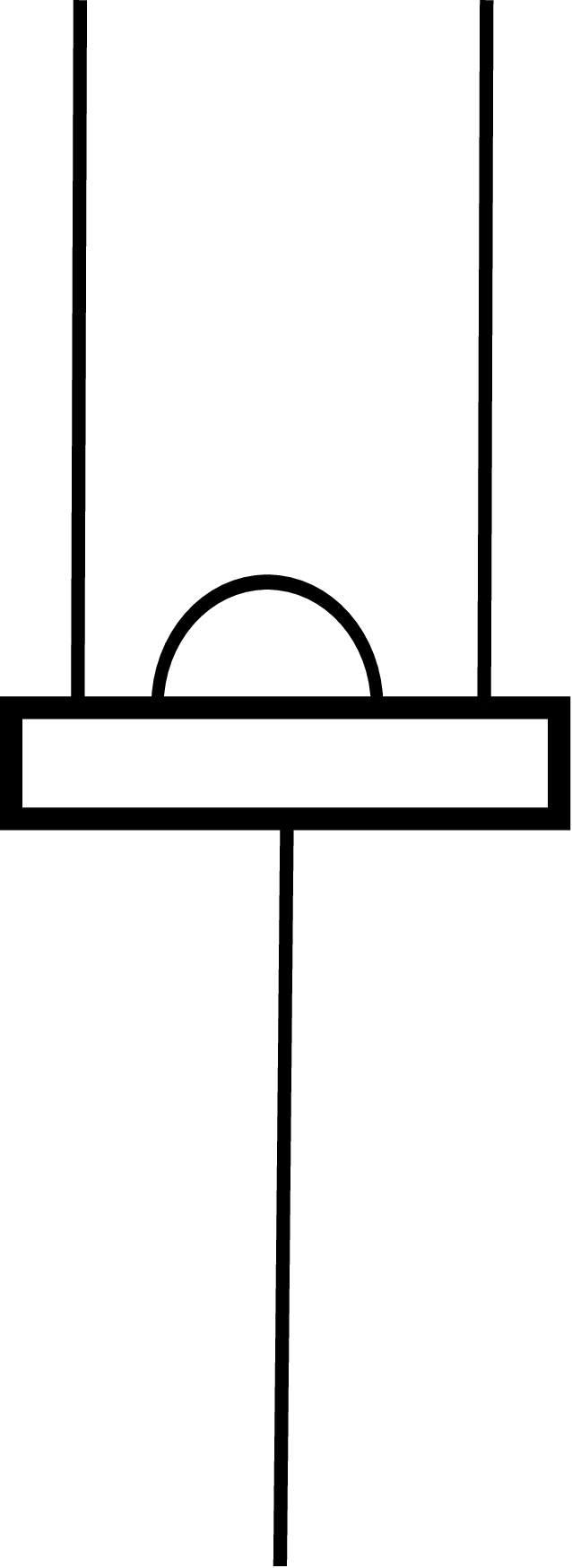}}
\put(-70,+82){\footnotesize{$n-i-2$}}
\put(-20,+64){\footnotesize{$1$}}
\put(-2,+82){\footnotesize{$i$}}
\put(-28,20){\footnotesize{$n$}}
\end{minipage}
=0.
\label{AX}
\end{eqnarray}
The second equation of \ref{AX} holds for $0\leq i\leq n-2$. We will also need the following properties of the projector:     
\begin{eqnarray}
\label{properties_2}
\hspace{8 mm}
\begin{minipage}[h]{0.08\linewidth}
\vspace{0pt}
\scalebox{0.115}{\includegraphics{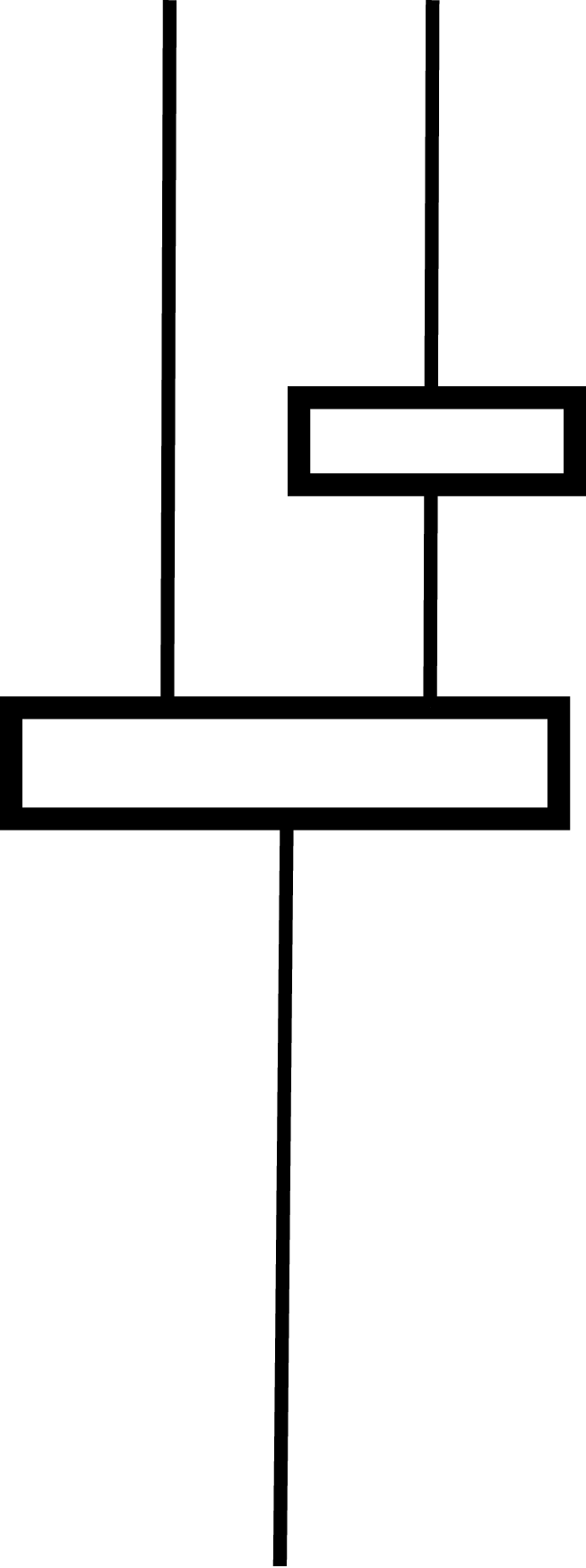}}
\put(-34,+82){\footnotesize{$n$}}
\put(-19,+82){\footnotesize{$m$}}
\put(-46,20){\footnotesize{$m+n$}}
\end{minipage}
=
\begin{minipage}[h]{0.09\linewidth}
\vspace{0pt}
\scalebox{0.115}{\includegraphics{idempotent2}}
\put(-46,20){\footnotesize{$m+n$}}
\end{minipage}
, \hspace{14 mm}
\begin{minipage}[h]{0.08\linewidth}
\vspace{-5pt}
\scalebox{0.19}{\includegraphics{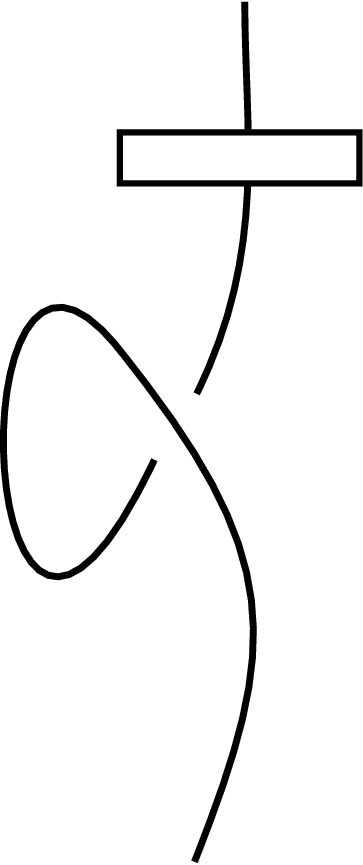}}
\put(-22,+75){\footnotesize{$n$}}
\end{minipage}
=(-1)^nA^{-n^2-2n}
\begin{minipage}[h]{0.09\linewidth}
\hspace{5pt}
\scalebox{0.19}{\includegraphics{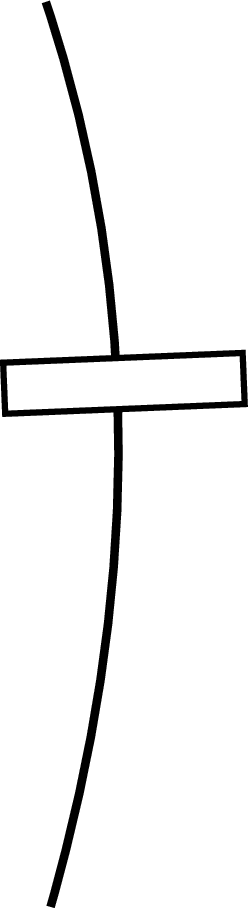}}
\put(-6,55){\footnotesize{$n$}}.
\end{minipage}
\end{eqnarray}
We also need the following fact from \cite{Hajij3}:
\begin{eqnarray}
\label{greatness}
   \begin{minipage}[h]{0.15\linewidth}
        \vspace{0 pt}
        \scalebox{0.4}{\includegraphics{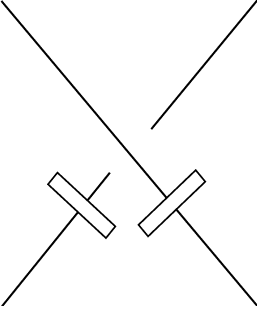}}
        \scriptsize{
         \put(-55,+52){$n$}
          \put(-1,+52){$n$}}

   \end{minipage}
   =
     \displaystyle\sum\limits_{i=0}^{n}C_{n,i}
  \begin{minipage}[h]{0.15\linewidth}
        \vspace{0pt}
        \scalebox{.4}{\includegraphics{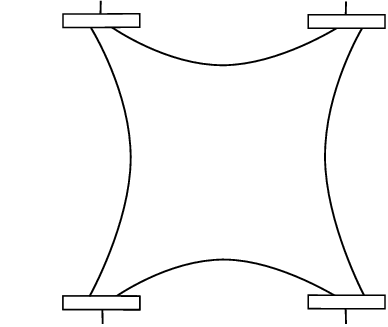}}
        \scriptsize{
        \put(-60,+64){$n$}
          \put(-5,+64){$n$}
          \put(-58,30){$i$}
          \put(-7,30){$i$}
         \put(-43,+56){$n-i$}
           \put(-43,+16){$n-i$}
           }
          \end{minipage}
  \end{eqnarray}
  where
  \begin{equation}
  \label{my fav}
  C_{n,i}=A^{n^2+2i^2-4in}\frac{(A^4;A^4)_n}{(A^4;A^4)_i(A^4;A^4)_{n-i}}.
  \end{equation}
Here $(a;q)_n$ is the $q$-Pochhammer which is defined as 
\begin{equation*}
(a;q)_n=\prod\limits_{j=0}^{n-1}(1-aq^j).
\end{equation*}  
The Jones-Wenzl projector can be used to extend the Kauffman bracket skein modules to banded trivalent graphs. Let $(a,b,c)$ a triple of positive integers. We say that $(a,b,c)$ is admissible if $a+b+c \equiv 0 \hspace{2pt}(\text{mod}\hspace{2pt} 2)$ and $|a-b|<c<a+b$. Given an admissible triple we define a $3$-valent vertex as :

 \begin{eqnarray}
 \label{untwist}
    \begin{minipage}[h]{0.19\linewidth}
        \vspace{-5pt}
        \scalebox{0.15}{\includegraphics{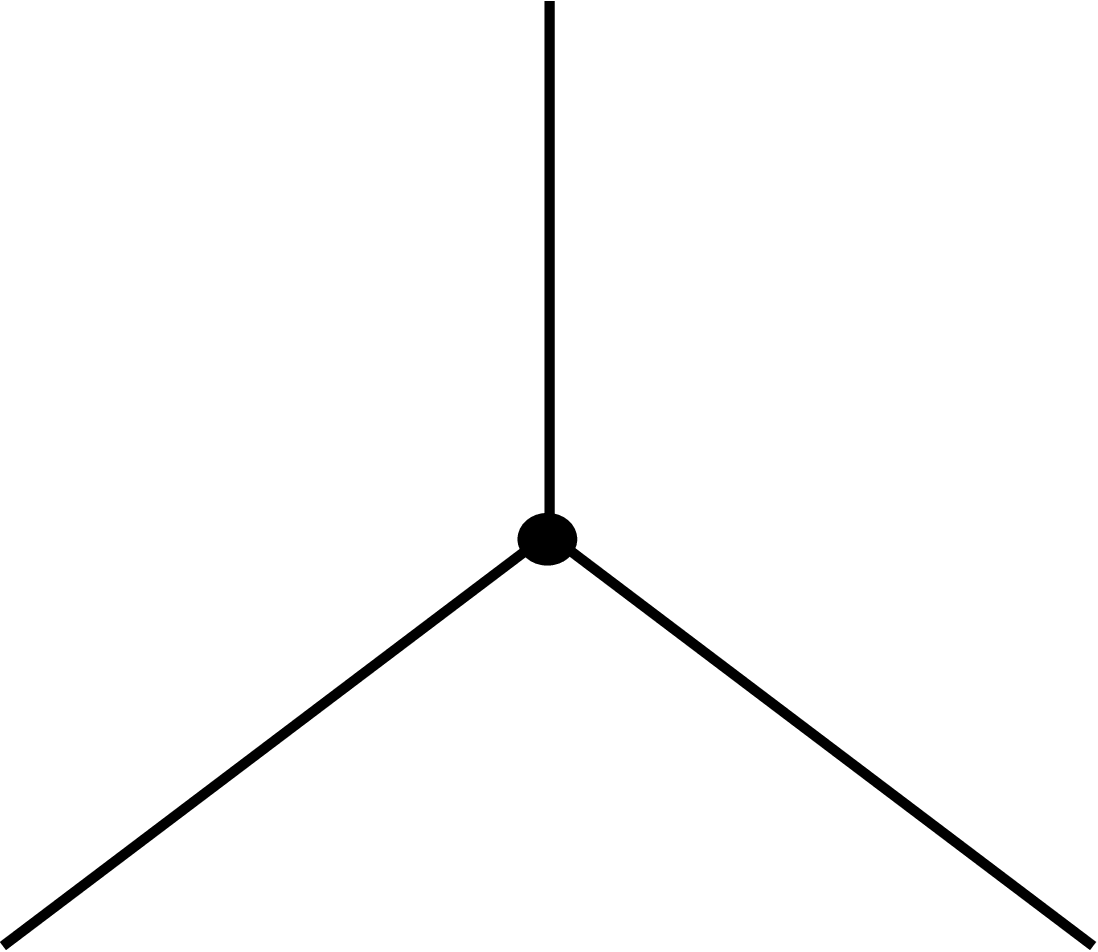}}
          \put(0,5){\footnotesize{$b$}}
        \put(-47,+60){\footnotesize{$a$}}
        \put(-77,+5){\footnotesize{$c$}}
   \end{minipage}
   =
\hspace{1 mm}  \quad
   \begin{minipage}[h]{0.1\linewidth}
        \vspace{0pt}
        \scalebox{0.4}{\includegraphics{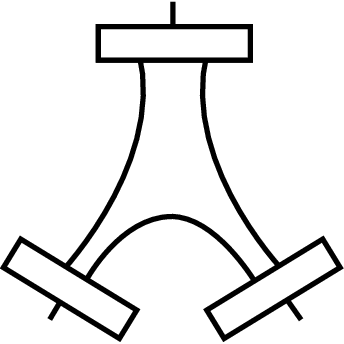}}
        \put(0,5){\footnotesize{$b$}}
        \put(-47,+65){\footnotesize{$a$}}
        \put(-70,+5){\footnotesize{$c$}}
		 \put(-19,40){\footnotesize{$x$}}
        \put(-47,40){\footnotesize{$y$}}
        \put(-40,14){\footnotesize{$z$}}        
        
   \end{minipage}\hspace{20pt}.
   \end{eqnarray}

Here $x$, $y$ and $z$ are three non-negative integers that are determined uniquely by the equations $a=x+y$, $b=x+z$ and $c=y+z$. These three equations are equivalent to the equations  
$x=\frac{a+b-c}{2}$,
$y=\frac{a+c-b}{2}$, and 
$z=\frac{b+c-a}{2}$ when the triple $(a,b,c)$ is admissible. We usually refer to the trivalent vertex shown on the right hand side of the previous equation by $\tau_{a,b,c}$ and refer to the positive integers $a,b$ and $c$ as \textit{colors}. Furthermore, we usually call $x$, $y$ and $z$ the \textit{internal colors} of the element $\tau_{a,b,c}$.\\ 

The \textit{fusion formula} is given by

 \begin{eqnarray}
 \label{fusion}
    \begin{minipage}[h]{0.09\linewidth}
        \vspace{0pt}
        \scalebox{0.4}{\includegraphics{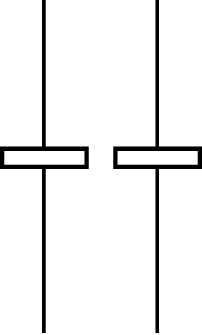}}
          \put(-5,+50){\footnotesize{$n$}}
        \put(-40,+50){\footnotesize{$n$}}
   \end{minipage}
   =\displaystyle\sum\limits_{i=0}^n\frac{\Delta_{2i}}{\theta(n,n,2i)}\hspace{1 mm}  
   \begin{minipage}[h]{0.09\linewidth}
        \vspace{0pt}
        \scalebox{0.3}{\includegraphics{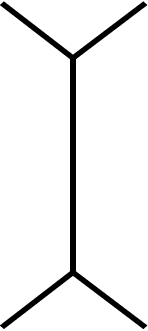}}
         \put(-29,+46){\footnotesize{$n$}}
        \put(-29,-1){\footnotesize{$n$}}
        \put(2,-1){\footnotesize{$n$}}
        \put(2,+46){\footnotesize{$n$}}
        \put(-5,+23){\footnotesize{$2i$}}
   \end{minipage}\hspace{10pt}
   \end{eqnarray}

%The previous formula implies
% \begin{eqnarray*}
%    \begin{minipage}[h]{0.09\linewidth}
%        \vspace{0pt}
%        \scalebox{0.4}%{\includegraphics{twoidempotent}}
%          \put(-5,+70){\footnotesize{$n$}}
%       \put(-44,+70){\footnotesize{$n$}}
%   \end{minipage}
%   &=&\displaystyle\sum\limits_{i=0}^n \frac{\Delta_{2i}}{\theta(n,n,2i)}\hspace{1 mm}  
%    \begin{minipage}[h]{0.06\linewidth}
%        \vspace{0pt}
%        \scalebox{0.3}{\includegraphics{Orthogonal_basis}}
%             \put(-5,+90){\footnotesize{$n$}}
%        	\put(-44,+90){\footnotesize{$n%$}}
%        	\put(-12,+55){\footnotesize{$i$}}
%        	\put(-32,+66){\tiny{$n-i$}}
%   \end{minipage}
% \end{eqnarray*} 

\noindent  
where $\theta(a,b,c)$ is the evaluation of the \textit{theta graph}, Figure \ref{thetagraph}, in the Kauffman bracket skein module of $S^2$. We will denote the trivalent graph that appears on the right handside of (\ref{fusion}) by $T_{n,i}$.
\begin{figure}[H]
	\centering
	{\includegraphics[scale=0.1]{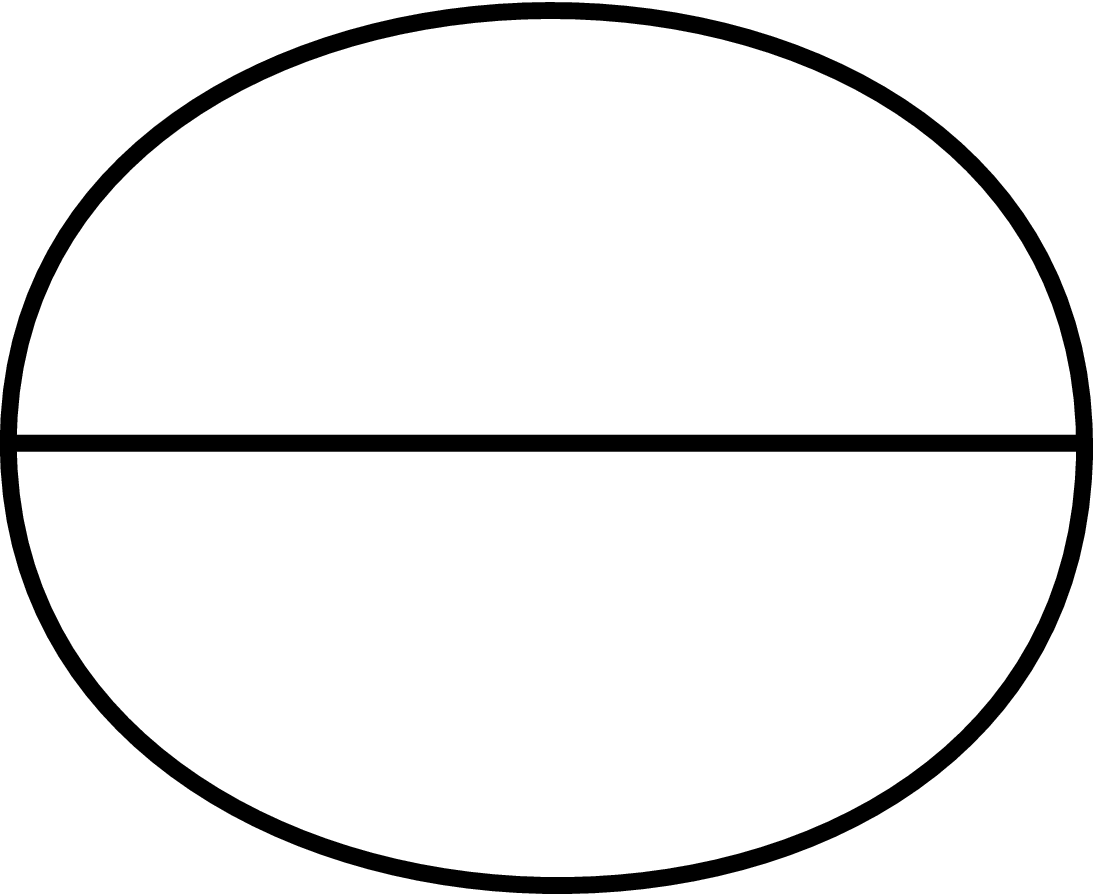}
		\put(-25,+45){\footnotesize{$a$}}
		\put(-25,+23){\footnotesize{$b$}}
		\put(-25,+4){\footnotesize{$c$}}
		\caption{The theta graph $\theta(a,b,c)$.}
		\label{thetagraph}}
\end{figure}

 This evaluation is given explicitly by 
\begin{equation}
\label{mn3}
\theta(a,b,c)=(-1)^{x+y+z} A^{-2(x+y+z)}  
 \frac{(A^4;A^4)_{x}(A^4;A^4)_{y}(A^4;A^4)_{z} (A^4 ;A^4)_{x+y+z+1}}{ (1-A^4)(A^4;A^4)_{x+y}(A^4;A^4)_{y+z}(A^4;A^4)_{x+z}},
\end{equation}
where $x$, $y$ and $z$ are the internal colored of the $3$-vertex $\tau_{a,b,c}$. Furthermore, one has:

 \begin{eqnarray}
 \label{untwist}
    \begin{minipage}[h]{0.09\linewidth}
        \vspace{0pt}
        \scalebox{0.1}{\includegraphics{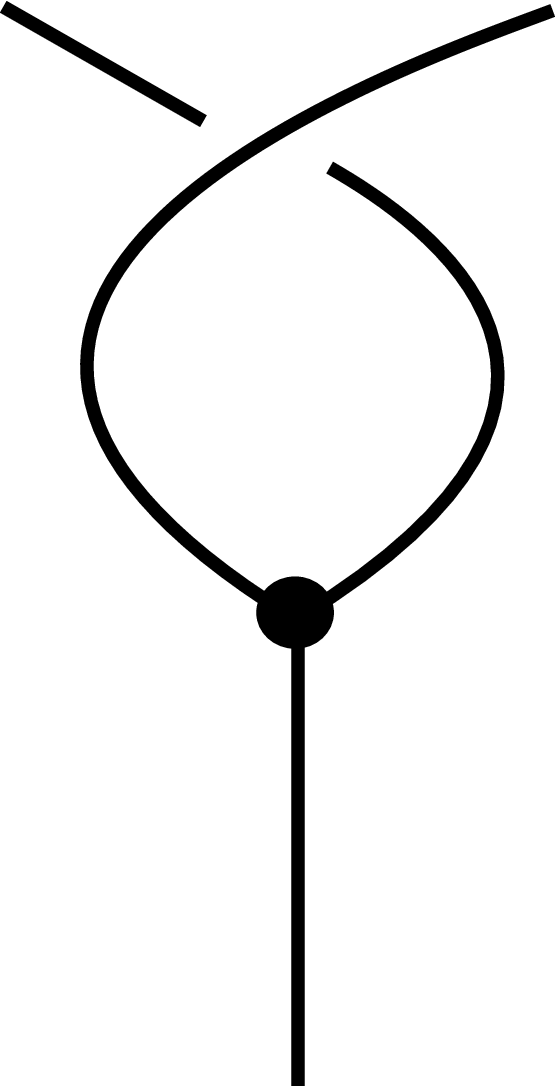}}
          \put(5,+50){\footnotesize{$b$}}
        \put(-40,+50){\footnotesize{$a$}}
        \put(-20,+5){\footnotesize{$c$}}
   \end{minipage}
   =\mu_c^{a,b}
\hspace{1 mm}  
   \begin{minipage}[h]{0.09\linewidth}
        \vspace{0pt}
        \scalebox{0.1}{\includegraphics{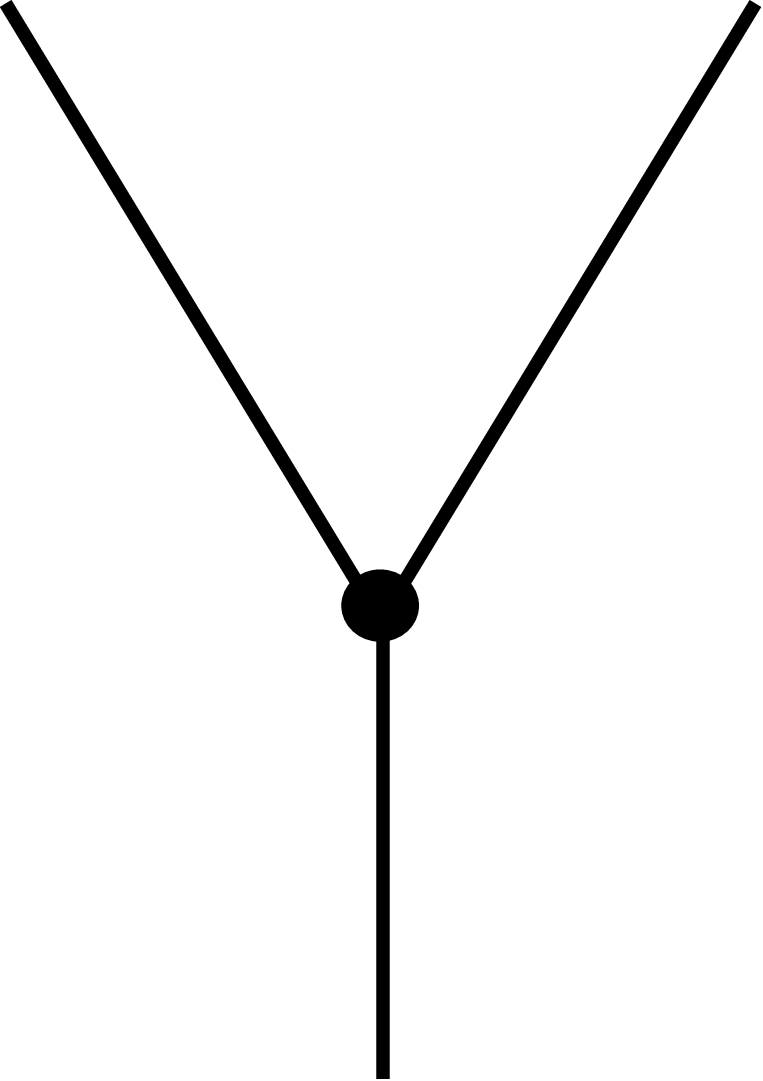}}
         \put(5,+50){\footnotesize{$b$}}
        \put(-45,+50){\footnotesize{$a$}}
        \put(-23,+5){\footnotesize{$c$}}
   \end{minipage}
   \end{eqnarray}
where $\mu_c^{a,b}=(-1)^{\frac{a+b-c}{2}}A^{
a+b-c +\frac{a^2+b^2-c^2}{2}}$.\\

\noindent
We will also need the definition of the colored Jones polynomial. Let $L$ be a framed link in $S^3$. Take cabling of every component of $L$, according to its framing, by the $n^{th}$ Jones-Wenzl idempotent and consider the evaluation of decorated framed link as an element of $\mathcal{S}(S^3)$. Up to a power of $\pm A$, that depends on the framing of $L$, the value of this element is defined to be the $n^{th}$ (unreduced) colored Jones polynomial polynomial $\tilde{J}_{n,L}(A)$. The colored Jones polynomial can be recovered from the unreduced colored Jones polynomial via:
\begin{equation}
\label{change of variable}
J_{n+1,L}(q)=\frac{\tilde{J}_{n,L}(A)}{\Delta_n}\bigg|_{A=q^
{1/4}}.
\end{equation}
Since we are interested in the list of the coefficients of the colored Jones polynomial, we do not need to take into consideration the framing of the link in our study. For this reason, in our computations of the coefficients of the $n$-th unreduced colored Jones polynomial of a link $L$ we choose any diagram $D$ of the link $L$ and compute the evaluation of the skein element obtained from $D$ by blackboard cabling all its components by the $n^{th}$ Jones-Wenzl idempotent. We denote this evaluation by $\langle S_n(D) \rangle$. Note that 
\begin{equation*}
\tilde{J}_{n,L} \doteq \langle S_n(D) \rangle.
\end{equation*}

\begin{remark}
	We will only state the facts that are concerned with the lowest terms of the colored Jones polynomial of alternating links. It should be noted, however, that these facts can also be proven for the highest terms analogously.
\end{remark}
\section{Coefficients Stability of the  Jones polynomial under Increasing Number of Twists}
\label{sec3}
In this section we give a proof of the coefficients stability of the Jones polynomial of alternating link diagram under increasing the number of twists in a maximal twist region. Later we will give a general proof for all colors. We choose to give a separate proof for the Jones polynomial since it is less technical and it illustrates clearly the main idea of the more general technical proof given in section \ref{sec4}.
\subsection{Alternating Links and the Minimal Degree of the Jones Polynomial} 
 Let $L$ be a link in $S^3$ and let $D$ be a link diagram of $L$. For any crossing in $D$ there are two ways to smooth it, the positive smoothing and the negative smoothing. See Figure \ref{smoothings111}.
\begin{figure}[H]
  \centering
   {\includegraphics[scale=0.14]{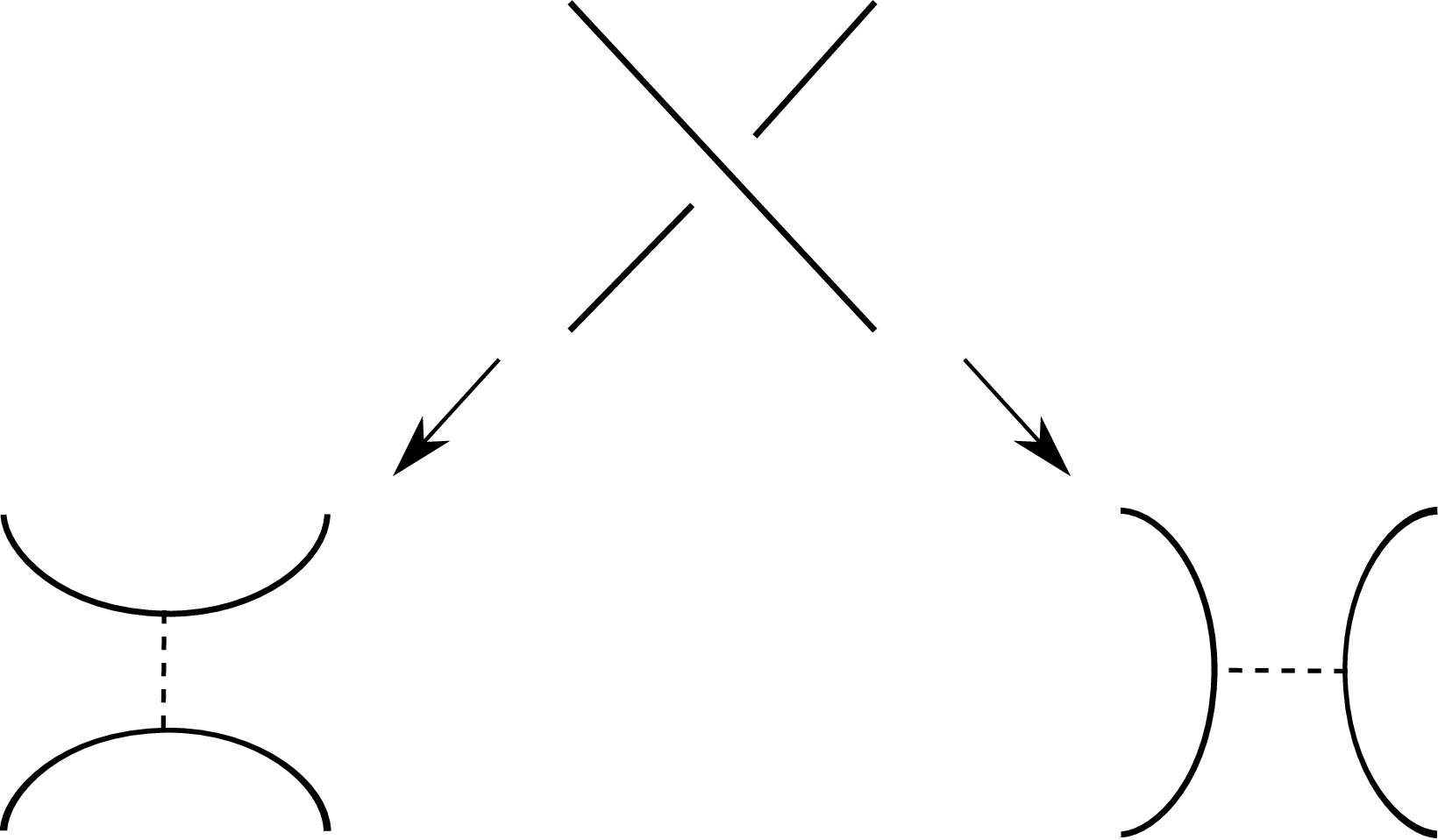}
    \put(-88,33){$+$}
          \put(-30,33){$-$}
    \caption{The positive and the negative smoothings of a crossing.}
  \label{smoothings111}}
\end{figure}
After applying a smoothing to each crossing in $D$ we obtain a planar diagram consisting of a collection of disjoint circles in the plane. We call this diagram along with the crossings assignments a \textit{state} of the diagram $D$. More precisely, write $R_D$ to denote the set of crossings of the diagram $D$. A state of the diagram $D$ is a function $s: R_D \longrightarrow \{-1,+1\}$. If $s$ is a state of $D$ then we denote by $s(D)$ the diagram constructed from $D$ using $s$. The state which assigns to every crossing the value $+1$ is called the all-positive state and is denoted by $s_+(D)$. The all-negative state $s_-(D)$ is defined similarly. For a state $s$ of a diagram $D$ we write $|s(D)|$ to denote the number of connected components in $s(D)$. 

The all-negative state can be used to compute the minimum degree of adequate links, a class of links that contain alternating links. We will need the notion of adequate links later so we give its definition here. 
\begin{definition}
 Let $D$ be a link diagram. The minus-graph of the diagram $D$, denoted $\mathbb{G}_{-}(D)$ is the graph whose vertices are the circles of $s_-(D)$ and whose edges correspond to the crossings in the diagram $D$. The reduced minus-graph of $D$, denoted by $\mathbb{G}^{\prime}_{-}(D)$, is obtained from $\mathbb{G}_{-}(D)$ by replacing parallel edges by a single edge (see Figure \ref{allB}).
\end{definition}
The plus-graph and the reduced plus-graph of a link diagram $D$ are defined similarly. A link diagram $D$ is called minus-adequate if  $\mathbb{G}_{-}(D)$ does not contain any loop. The notion of plus-adequate diagram is defined similarly. A link diagram is adequate if it is both minus-adequate and plus-adequate. It is known \cite{Th} that a reduced alternating diagram is adequate.\\

 \begin{figure*}[htb]
 	\centering
 	{\includegraphics[scale=0.1]{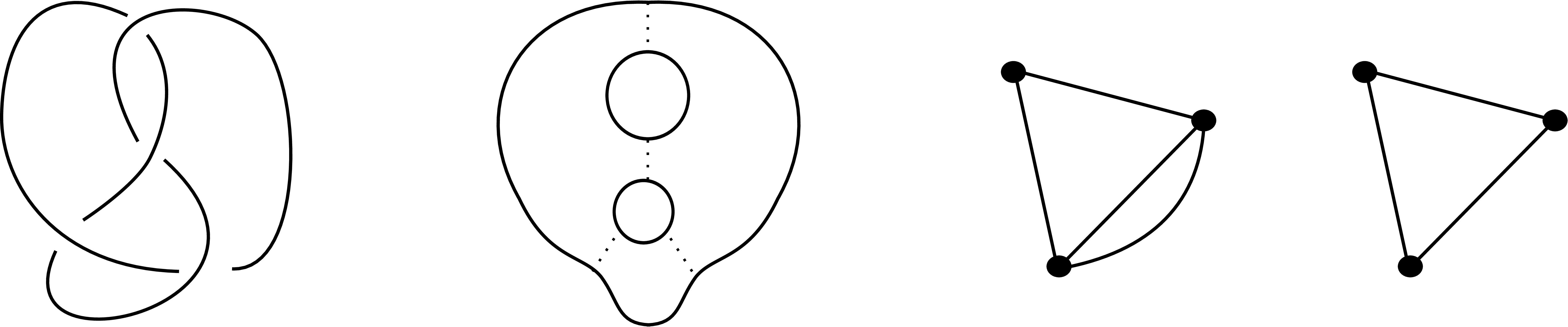}
 		\put(-290,-15){$D$ }
 		\put(-10,-15){$\mathbb{G}^{\prime}_-(D)$  }
 		\put(-195,-15){$s_-(D)$ }
 		\put(-100,-15){$\mathbb{G}_-(D)$ }
 		\caption{A link diagram $D$, its all-negative state $s_-(D)$, the minus-graph $\mathbb{G}_-(D)$, and the reduced all minus-graph $\mathbb{G}^{\prime}_-(D)$.  }
 		\label{allB}}
 \end{figure*}
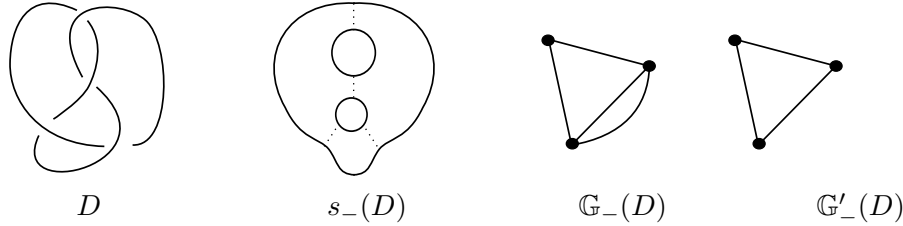
The Kauffman skein relation can be used to show that the all-positive state (respectively
the all-negative state) realizes the highest (respectively the lowest) coefficient of the Jones polynomial of a plus-adequate (respectively minus-adequate) link.

For an element $f$ in $\mathcal{S}(S^2)$ expressed as an element in $\mathbb{Z}[A^{-1}][[A]]$, the symbol $m(f)$ will denote the minimum degree of $f$.

\begin{proposition}\cite{Th}
Let $D$ be a minus-adequate link diagram with $c$ crossings. Then
\begin{equation*}
\label{mainprop1}
m( \langle D \rangle )=m(\langle s_-(D)\rangle )=-c-2|s_-(D)|.
\end{equation*}
and for any state $s$ of $D$ different from $s_-$, we have 
\begin{equation*}
m(\langle s(D)\rangle )-4 \geq m(\langle s_-(D) \rangle ), 
\end{equation*}
with equality if the state $s$ is obtained from $s_-$ by changing the label of one crossing from a positive smoothing to a negative smoothing.
\end{proposition}

Our work here can be considered as a generalization of the previous fact. We are interested in the list of coefficients of the Jones polynomial. For this reason we need the following definition.

Let $P_1(q)$ and $P_2(q)$ be power series in $\mathbb{Z}[q^{-1}][[q]]$. For a non-negative integer $n$, we say that $P_1$ and $P_2$ are \textit{ $n$ equivalent} and write $P_1(q)\doteq_n P_2(q)$, if their first $n$ coefficients agree up to a sign. For instance, $-q^{-4} + 4q
^{-3} - 6 + 11 q \doteq_5 1 - 4q + 6q^4$. If $P_1(q) \doteq_n P_2(q)$ for every integer $n \geq 0$, then we simply write $P_1(q) \doteq P_2(q)$.
\begin{definition}
\label{stab}
	Let $\mathcal{P}=\{P_n(q)\}_{n \in \mathbb{N}}$ be a sequence of formal power series in $\mathbb{Z}[q^{-1}][[q]]$ and let $f: \mathbb{N} \longrightarrow \mathbb{N}$ be an increasing function. We say that the sequence $\mathcal{P}$ stabilizes with rate $f$  if there exists a formal power series $T_{\mathcal{P}}(q)$ in $\mathbb{Z}[[q]]$ that satisfies
	\begin{equation*}
	\label{main defintion}
	T_{\mathcal{P}}(q)\doteq_{f(n)}P_n(q), \text{ for all } n \in \mathbb{N}.
	\end{equation*}
	\end{definition}
	
Note that the sequence $\mathcal{P}$ stabilizes with rate $f$ if and only if $P_n(q)\doteq_{f(n)}P_{n+1}(q)$ for all $n \geq 1$.  We call the function $f$ the \textit{rate of stability} of the sequence $\mathcal{P}$ and we call the power series $T_{\mathcal{P}}(q)$ the \textit{tail} of $\mathcal{P}$. If the tail $T_{\mathcal{P}}(q)$ of a sequence $\mathcal{P}$ exists then it is independent from the rate of stability. Note that if $\mathcal{P}$ stabilizes with rate $f$ then it also stabilizes with any rate $g$ such that $g \leq f$. We say that the rate of stabilization $f$ is maximal for the sequence $\mathcal{P}$ if stabilizes with rate $f$ but it does not stabilize for any rate $g$ such that $g > f$.

 We will show that the Jones polynomial of a sequence of alternating links parametrized by the number of twists in a maximal negative twist regions has a well-defined tail in the sense of the Definition \ref{stab}. For this purpose we need the following notion.	
	
\subsection{Twist Regions}
Let $L$ be a link diagram. Suppose that $L$ has $r$ maximal negative twist regions labeled by $1, \ldots, r$. Let $k_i$ be the number of negative crossings in the region $i$. We will denote by $L(k_1,\ldots,k_r)$ the link diagram $L$ with $r$ labeled maximal negative twist regions such that the $i^{th}$ region has $k_i \geq 1$ crossings. In the case when we are interested in a subset of the total sets of $r$ twist regions we will only label these regions that we are interested in and use the same notation above to denote the link diagram with the labeled twist regions. In particular, when we are interested merely in a single maximal negative twist region with $k$ crossings in the diagram $L$ then we will refer to this link diagram by $L_{k}$ and refer to the twist region we are interested in as \textit{the marked twist region} of the diagram $L$. Figure \ref{example0} shows a link diagram with a total of $4$ maximal negative twist regions labeled by two different methods.
 \begin{figure}[H]
	\centering
	{\includegraphics[scale=0.23]{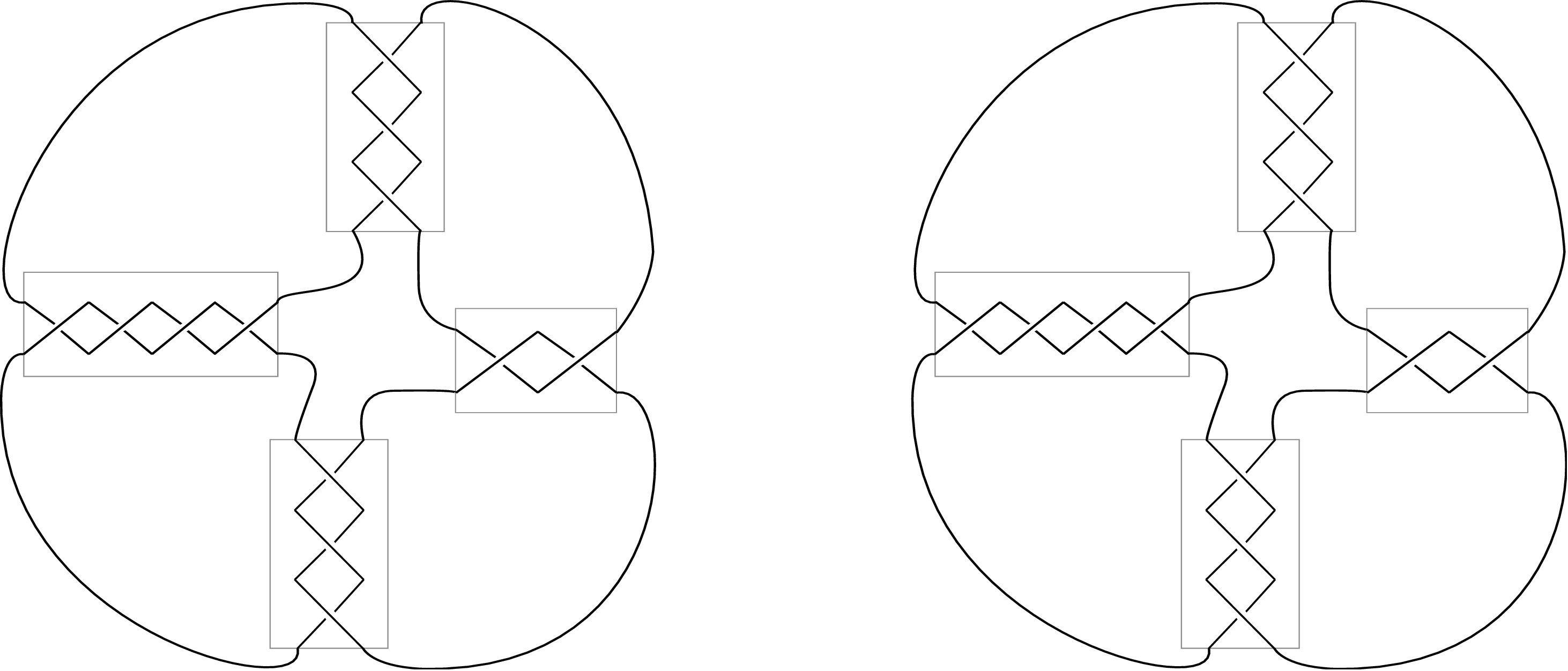}
	\put(-340,+93){\footnotesize{$L^{\prime}$}}
	\put(-143,+93){\footnotesize{$L$}}
		\put(-300,+83){\footnotesize{$1$}}
		\put(-210,75){\footnotesize{$2$}}
		\put(-275,8){\footnotesize{$3$}}
		\put(-125,+83){\footnotesize{$1$}}
		\put(-75,+100){\footnotesize{$2$}}
		\put(-35,75){\footnotesize{$3$}}
		\put(-85,8){\footnotesize{$4$}}
		\caption{The link $L^{\prime} =L^{\prime}(4,2,3)$ on the left and the link $L=L(4,3,2,3)$ on the right.}
		\label{example0}}
\end{figure}
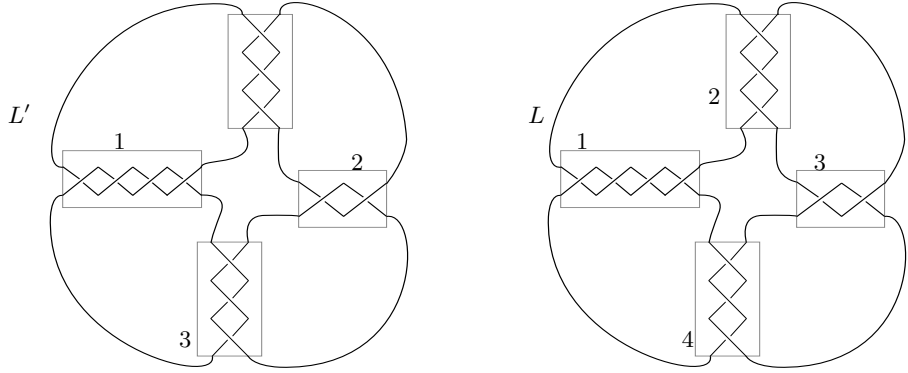
Let $L=L(k_1,\ldots,k_r)$ be an alternating link diagram such that $k_i \geq 1$. Suppose that $b_i \in \mathbb{Z}$ such that  $k_i+b_i \geq 0$. The link diagram $L(k_1+b_1,\ldots,k_r+b_r)$ is the link diagram obtained from $L(k_1,\ldots,k_r)$ by replacing the i-th twist region which has $k_i$ negative crossings with a twist region with $k_i+b_i$ negative crossings. In the case when $k_i+b_i=0$ then we replace the i-th maximal twist region by the negative smoothing as illustrated in Figure \ref{replacement}.

 \begin{figure*}[htb]
 	\centering
 	{\includegraphics[scale=0.4]{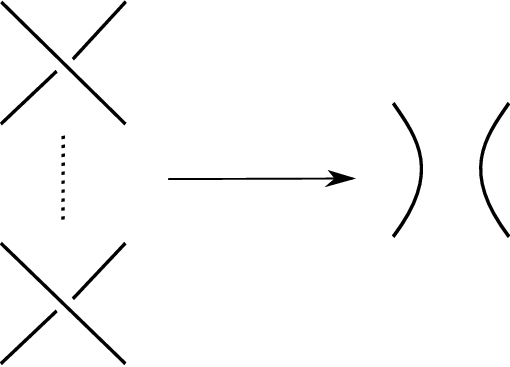}
 		\caption{Replacing a maximal twist region by the negative smoothing when the number of this twist region becomes zero.  }
 		\label{replacement}}
 \end{figure*}

\begin{theorem}\label{KauffAlt}
\label{new greatness}
Let $L=L_k$ be an alternating link diagram with a marked maximal negative twist region with $k \geq 1$ crossings. Then,
\begin{equation}
 \langle L_{k} \rangle  \doteq_{4k} \langle L_{k-1} \rangle.
\end{equation} 
\end{theorem}	
\begin{proof}
Using the Kauffman skein relation to smooth one of the crossings in the marked twist region gives:

   \begin{eqnarray*}\langle L_{k} \rangle &=&  A \quad \begin{minipage}[h]{0.1\linewidth}
        \vspace{0pt}
        \scalebox{0.18}{\includegraphics{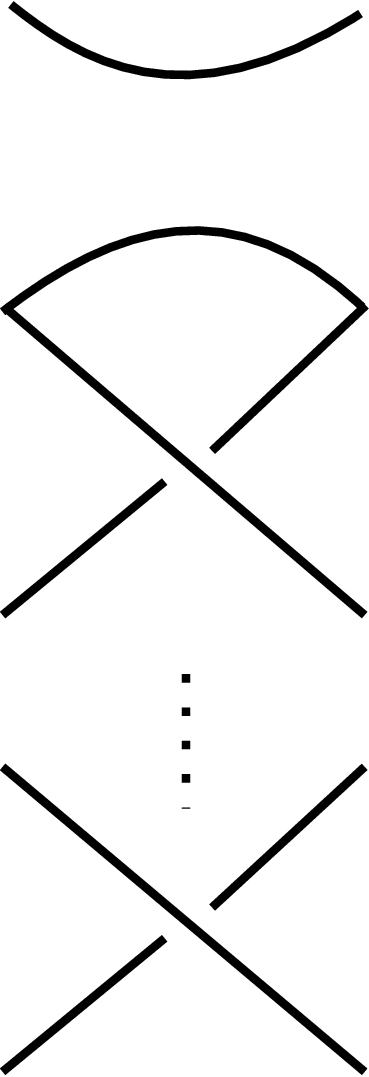}}
        \put(-42,+30){\footnotesize{$k-1$}}
   \end{minipage}
   + A^{-1}   
    \begin{minipage}[h]{0.01\linewidth}
        \vspace{0pt}
        \scalebox{0.18}{\includegraphics{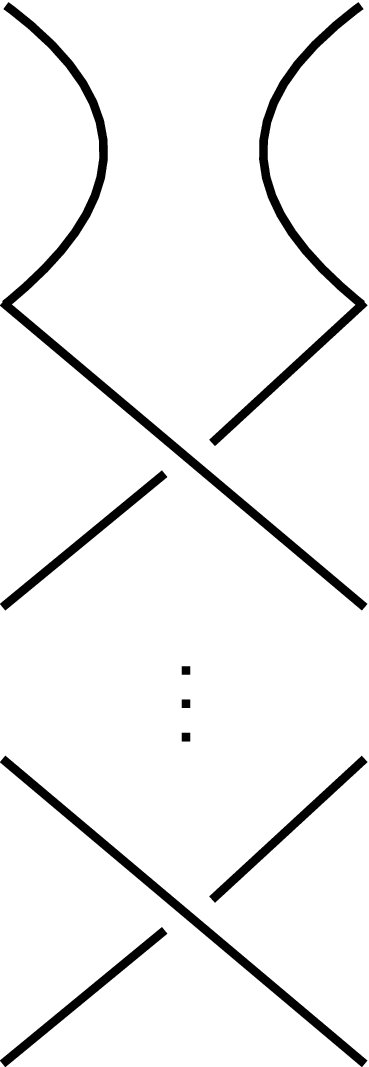}}
             \put(-0,+30){\footnotesize{$k-1$}}
   \end{minipage}
 \end{eqnarray*}
Untwisting the regions in the first term in the previous equation we obtain: 

\begin{eqnarray}
\label{mainequation}
 \langle L_{k} \rangle &=& A (-A^3)^{k-1} \quad \begin{minipage}[h]{0.1\linewidth}
        \vspace{0pt}
        \scalebox{0.18}{\includegraphics{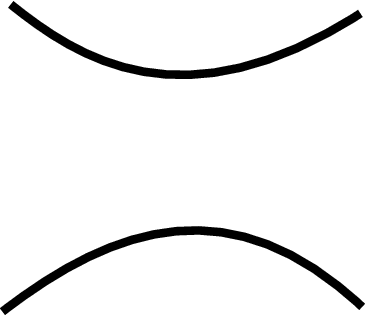}}
   \end{minipage} + A^{-1} \langle L_{k-1} \rangle. 
 \end{eqnarray}  
  
Let $L^{\prime}$ be the link obtained from $L_k$ in the first term in equation \ref{mainequation}. Figure \ref{smoothings1} shows that both $L^{\prime}$ and $L_{k-1}$ are alternating.
\begin{figure}[H]
  \centering
   {\includegraphics[scale=0.14]{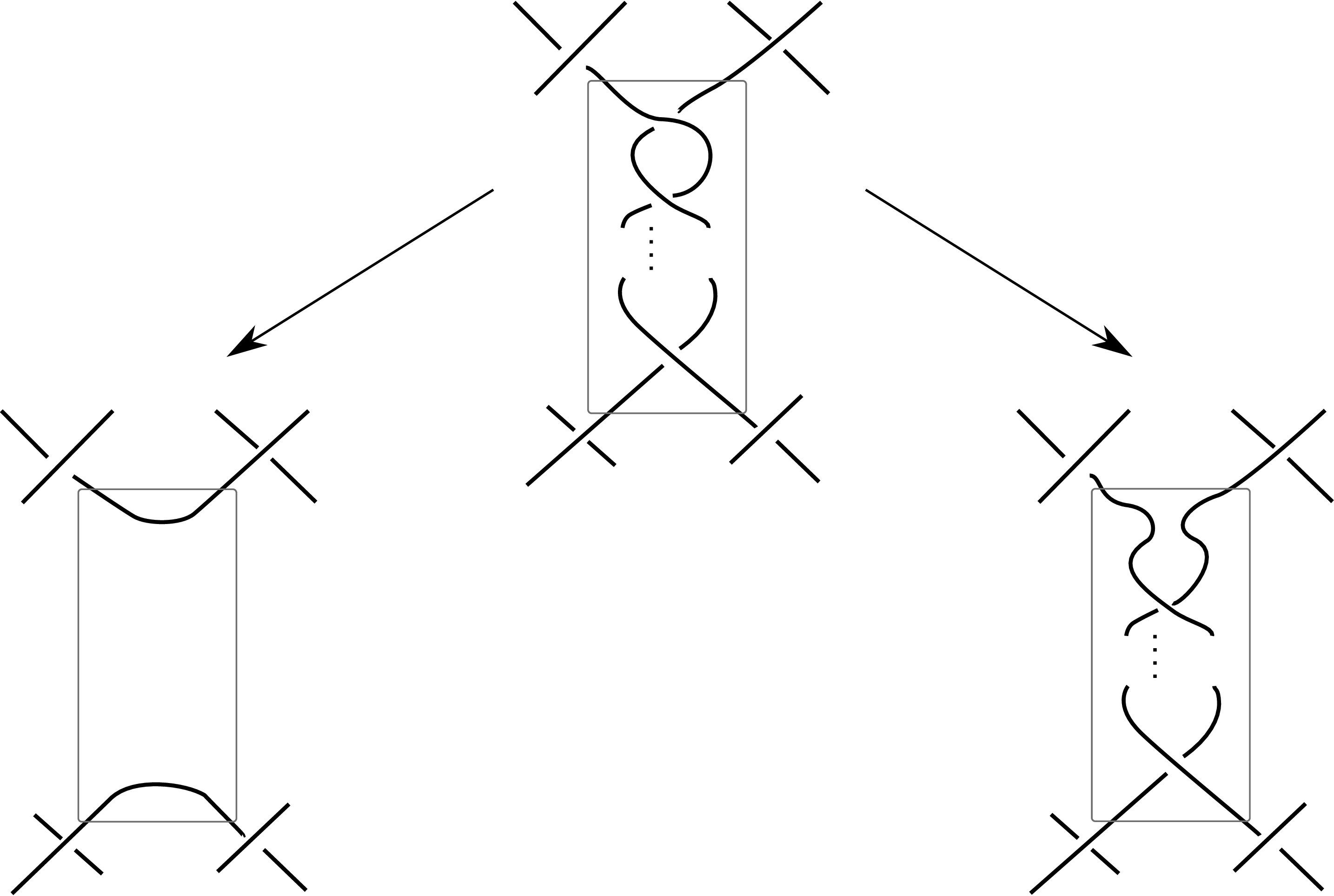}
    \put(-188,83){$L^{\prime}$}
          \put(-30,83){$L_{k-1}$}
    \caption{Giving a maximal negative twist region in an alternating link $L=L_k$, the two links $L^{\prime}$ and $L_{k-1}$ obtained from $L$ are alternating.}
  \label{smoothings1}}
\end{figure} 
Assuming that the total number of crossings in the link diagram $L_k$ is $c$, then the numbers of crossings in the link diagrams $L_{k-1}$ and $L^{\prime}$ are $c-1$ and $c-k$ respectively. Furthermore, it is clear that $|s_{-}(L_k)|=|s_{-}(L_{k-1})|$ and $|s_{-}(L^{\prime})|=|s_{-}(L_k)|-1$.

Thus by Proposition \ref{mainprop1}: 

\begin{eqnarray*} m ( A (-A^3)^{k-1} \langle L^{\prime} \rangle)&=&(3k-2)-(c-k)-2(|s_{-}(L_k)|-1)\\&=&4k-c-2|s_{-}(L_k)|,
 \end{eqnarray*}  
 and
\begin{eqnarray*} m ( A^{-1} \langle  L^{k-1} \rangle )&=&(-1)-(c-1)-2|s_{-}(L_k)|)\\&=&-c-2|s_{-}(L_k)|=m(L_k).
 \end{eqnarray*} 
Hence, by equation (\ref{mainequation}), the first $4k$ terms of $A^{-1} \langle L_{k-1} \rangle$ do not get canceled by any term in the expansion of $-A^{3k-2} \langle L^{\prime} \rangle$. This completes the proof.   
\end{proof}
Formulas for the Kauffman bracket under twisting can be also found in \cite{CK1,CK2}. Equation (\ref{change of variable}) and the previous theorem \ref{KauffAlt} imply immediately the following.

\begin{corollary}
\label{1}
Let $L=L_k$ be an alternating link diagram with a marked maximal negative twist region with $k \geq 1$ crossings. Then the sequence $\{J_2(L_{k+i}) \}_{i \in \mathbb{N}}$ has a well-defined tail.
\end{corollary}

Now assume that an alternating link $L=L(k_1,k_2)$ with $k_1,k_2 \geq 1 $ has two marked maximal twist regions and we want to compare the first few coefficients of the Jones polynomial $L(k_1-1,k_2-1)$ with $L=L(k_1,k_2)$. By the previous theorem one has
\begin{equation*}
\langle L(k_1,k_2)\rangle \doteq_{4k_1} \langle L(k_1-1,k_2)\rangle \doteq_{4k_2} \langle L(k_1-1,k_2-1)\rangle .
\end{equation*}
Hence,
\begin{equation*}
\langle L(k_1,k_2)\rangle \doteq_{4k} \langle L(k_1-1,k_2-1) \rangle, 
\end{equation*}
where $k:=\min(k_1,k_2)$. In general we have the following.
\begin{theorem}
\label{2}
Let $L=L(k_1, \ldots,k_r)$ be an alternating link diagram. Then
\begin{equation}
\langle L(k_1,\ldots, k_r)\rangle \doteq_{4k} \langle L(k_1-1,\ldots, k_r-1) \rangle,
\end{equation}
where $k=\min_{1 \leq i \leq r }(k_i)$.
\end{theorem}
The previous theorem along with equation (\ref{change of variable}) immediately imply the following Corollary.

\begin{corollary}
\label{3}
Let $L=L(k_1,\ldots,k_r)$ be an alternating link diagram. Then the sequence $\{J_2(L(k_1+i,\ldots,k_r+i)) \}_{i \in \mathbb{N}}$ has a well-defined tail.
\end{corollary}
We will denote by $T_{2,L(k_1,\ldots,k_r)}$ the tail in Corollary \ref{3} associated with the alternating link diagram $L(k_1,\ldots,k_r)$ and the Jones polynomial.

The following Corollary follows from Theorem \ref{new greatness}.

\begin{corollary}\label{rate}
Let $L=L(k_1,\ldots,k_r)$ be an alternating link diagram. Then the sequence $\{J_2(L(k_1+i,\ldots,k_r+i)) \}_{i \in \mathbb{N}}$ stabilizes with rate $k+i+1$ where $k=\min_{1 \leq j \leq r }(k_j)$.
\end{corollary} 
The rate of stability for the sequence specified in Corollary \ref{rate} is maximal. This can be seen by considering the example of $J_2(P(8,6,i)$ where $P(c_1,c_2,c_3)$ is shown in Figure \ref{example0}. The following table shows that $J_2(P(8,6,i) \doteq_{i+1} J_2(P(8,6,i+1))$ for $1 \leq i \leq 3$, but this is not the case for the rate $i+2$.
\\

 \begin{tabular}{c|>{\raggedright\arraybackslash$}p{6.5cm}<{$}c}
 \hline
 \text{The link $P(8, 6, i)$ } & \text{List of lowest $i+2$ of coefficients of $J_2({P(8, 6, i)})$}\\
  \hline
$i=1$ & $1,-1,2$  \\
$i=2$ & $1,-1,3,-3$  \\ 
$i=3$ & $1,-1,3,-4,5$ \\
\hline 
\end{tabular}

\section{The Main Theorems}\label{sec4}
\subsection{The Colored Kauffman Skein Relation and the Minimal Degree of the Colored Jones Polynomial}
The minimum degree of the colored Jones polynomial of alternating links can be computed from the link diagram. For the purpose of this paper we need to state this fact in terms of the identity (\ref{greatness}). This identity (\ref{greatness}) generalizes the Kauffman skein relation. Motivated by this fact, we define the $n+1$ different states for a link diagram $D$ for every positive
integer $n$. More precisely, an $n$-colored state is a function $s^{n}: R_D \longrightarrow \{0,\ldots,n\}$. If $s^{n}$ is a state of $D$ then define $s^{n}(D)$ to be the skein element in $\mathcal{S}(S^2)$ obtained from $D$ by replacing every crossing labeled $0\leq i \leq n $ by the skein element shown in Figure \ref{general smoothing}. 
\begin{figure}[H]
	\centering
	{\includegraphics[scale=0.8]{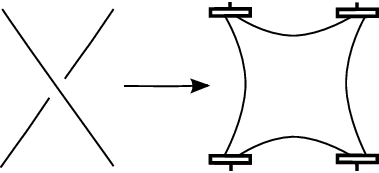}
		\put(-39,+61){\footnotesize{$n-i$}}
		\put(-10,+26){\footnotesize{$i$}}
		\caption{Given $n \geq 1$, there are $n+1$ colored smoothings of a crossing. }
		\label{general smoothing}}
\end{figure}
Furthermore, when it is necessary to specify the assignments of a state $s^{n}$ in the notation we will use $s^{n}(i_1,\ldots,i_k)$ to refer to $s^{n}$ where $ 0 \leq i_1,\ldots,i_k \leq n$ are the assignments for the crossings of $D$ defined by the state $s^{n}$. We will refer to the state which assigns $n$ to every crossing by $s^{n}_-$ and we will refer to the state which assigns $0$ to every crossing by $s^{n}_+$. For a state $s^{n}(i_1,\ldots,i_k)(D)$ of a diagram $D$ crossings define the skein element
\begin{equation}
\langle s^{n}(i_1,\ldots,i_k)(D) \rangle =\prod_{j=0}^k C_{n,i_j} s^{n}(i_1,\ldots,i_k)(D). 
\end{equation}
 It is clear from (\ref{greatness}) that
\begin{equation}
\langle S_n(D) \rangle=\sum_{s^n}  \langle s^{n}(D) \rangle.
\end{equation}
Using the identity (\ref{greatness}), the following was proved in \cite{Hajij3}.  
\begin{proposition} \cite{Hajij3}
\label{general_min}
Let $D$ be a reduced alternating link diagram with $c$ crossings. Then
\begin{equation}
m(\langle S_n(D) \rangle)=m(\langle s^{n}_- (D)\rangle)=-cn^2-2ns_-(D),
\end{equation}
and for any state $s^{n}$ different from $s^{n}_-$, we have
\begin{equation}
m(\langle s^{n}(D)\rangle)-2n \geq m(\langle s^{n}_- (D)\rangle),
\end{equation}
with equality if $s^{n}$ is obtained from $s^{n}_-$ by changing the label of one crossing from $n$ to $n-1$.
\end{proposition}
The proof of the previous fact utilizes finding the minimum degree of certain skein elements called \textit{adequate skein elements} due to Armond \cite{Armond}. We introduce this concept here since it will be needed later. Let $S$ be a crossingless skein element in $\mathcal{S}(S^2)$ consisting with circles and arcs connecting Jones-Wenzl projectors with various colors. Denote by $\bar{S}$ to the skein element obtained from $S$ by replacing every $i^{th}$ Jones-Wenzl idempotent with the identity element in $TL_{i}$. The skein element $S$ is called \textit{adequate} if $\bar{S}$ consists of circles each of which passes at most once through the regions where we had the boxes of the idempotents in $S$. Denote $M(S) := m(\bar{S})$. Computing the minimum degree of adequate skein elements can be done easily using the following Lemma:
\begin{lemma}
\cite{Armond}
\label{ArmondL}
Let $S$ be an in element in $\mathcal{S}(S^2)$ expressed as a single diagram containing the Jones-Wenzl idempotent, then $m(S) \geq M(S)$. Furthermore, if $S$ is adequate then $m(S)= M(S)$.
\end{lemma} 
\subsection{Twist Regions and Lowest Terms of the Colored Jones Polynomial}
We want to study the list of lowest terms of the colored Jones polynomial of a reduced alternating link diagram $L=L(k_1,\ldots ,k_r )$ as we increase the number of crossings in the labeled maximal twist regions. For this purpose we need to study the minimum degree of a certain trivalent graph obtained from the link diagram $L(k_1,\ldots ,k_r )$. Let $\Upsilon(n,p)$ be the skein element in $\mathcal{S}(S^2)$ obtained from $L(k_1,\ldots,k_r)$ by cabling all the strands outside of the first maximal twist region by the $n^{th}$ Jones-Wenzl projector and replacing this twist region by the trivalent graph $T_{n,p}$, where $0\leq p < n$, as illustrated in Figure \ref{replace}.

\begin{figure}[H]
	\centering
	{\includegraphics[scale=0.2]{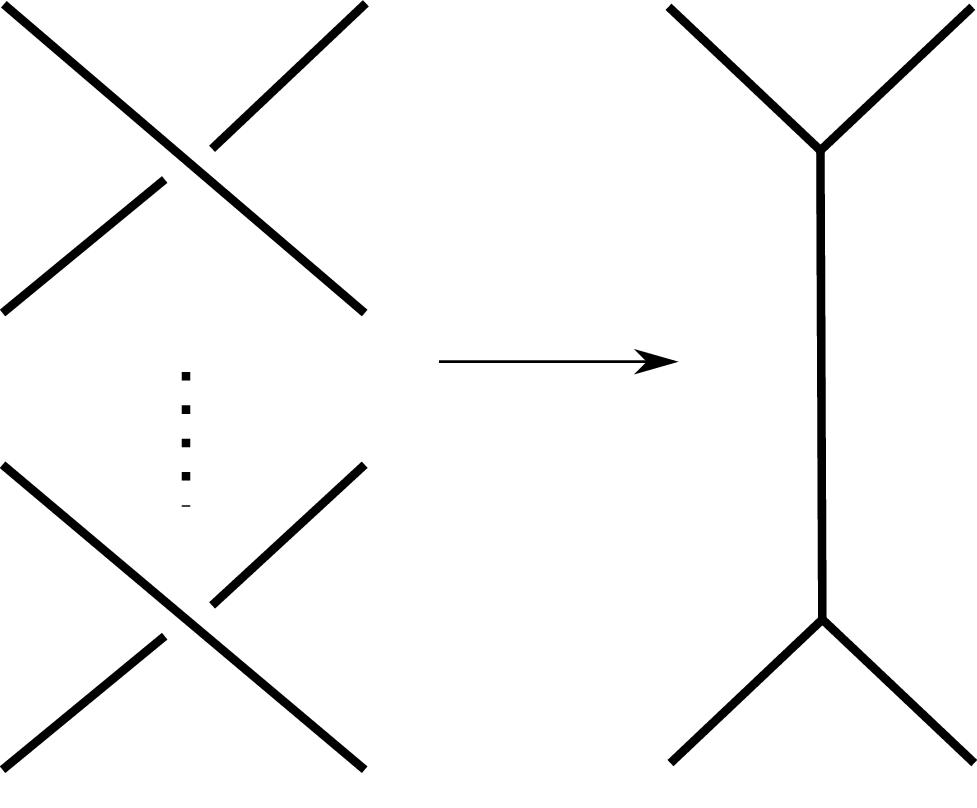}
		\put(-35,+69){\footnotesize{$n$}}
		\put(0,+69){\footnotesize{$n$}}
		\put(-35,8){\footnotesize{$n$}}
		\put(0,8){\footnotesize{$n$}}
		\put(-98,+37){\footnotesize{$k_1$}}
		\put(-10,+35){\footnotesize{$2p$}}
		\caption{Replacing a maximal twist region with $k_1$ crossings in the alternating diagram $D$ by the trivalent graph $T_{n,p}$. }
		\label{replace}}
\end{figure}See also Figure \ref{example} for an example.
\begin{figure}[H]
	\centering
	{\includegraphics[scale=0.1]{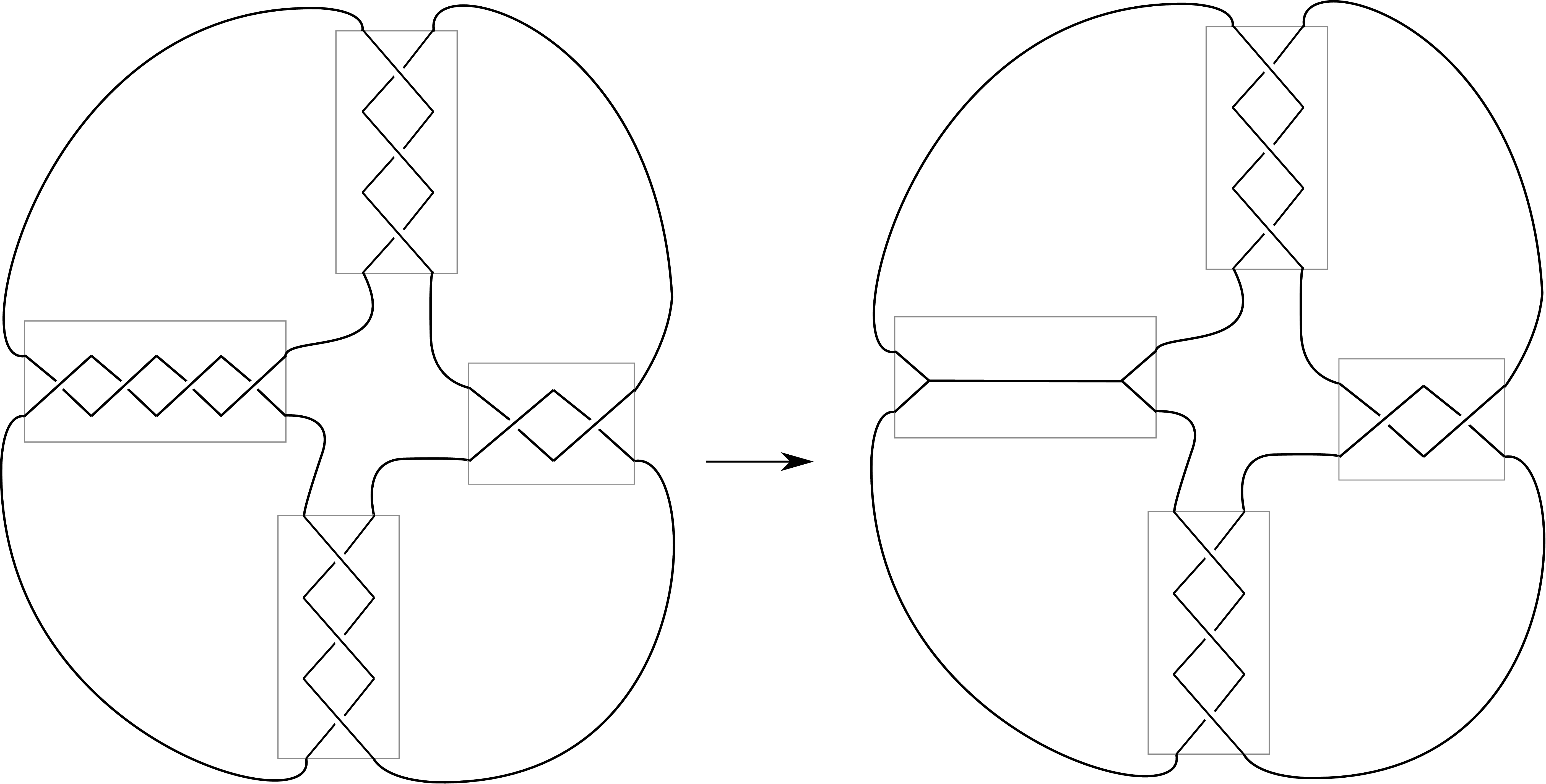}
		\put(-235,+83){\footnotesize{$1$}}
		\put(-185,+100){\footnotesize{$2$}}
		\put(-165,75){\footnotesize{$3$}}
		\put(-198,8){\footnotesize{$4$}}
		\put(-100,74){\footnotesize{$2p$}}
		\put(-100,130){\footnotesize{$n$}}
		\put(-100,15){\footnotesize{$n$}}
		\put(0,15){\footnotesize{$n$}}
		\put(-8,115){\footnotesize{$n$}}
		\caption{On the left $L=L(4,3,2,3)$ where the labels $1,2,3$ and $4$ for the maximal negative twist regions are shown on the diagram. On the right the skein element $\Upsilon(n,p)$ is obtained by replacing the twist region labeled $1$ by the trivalent $T_{n,p}$, where $0\leq p < n$.}
		\label{example}}
\end{figure}
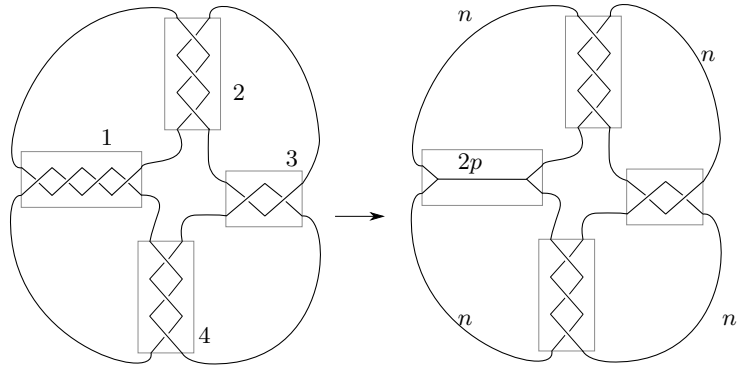
Applying the fusion formula to the twist regions $2,\ldots,r$ in the skein element $\Upsilon(n,p)$ we obtain,
\begin{equation}
\label{newc}
\Upsilon(n,p)=\sum_{0 \leq j_2, \ldots, j_r \leq n}\prod_{i=2}^{r} (\mu_{2j_i}^{n,n})^{k_i}\frac{\Delta_{2j_i}}{\theta(n,n,2j_i)}\Gamma_{n,p,(j_2,\ldots,j_r)},
\end{equation}
where $\Gamma_{n,p,(j_2,\ldots,j_r)}$ is the trivalent graph obtained from $\Upsilon(n,p)$ by replacing $i$-th twist region by an edge colored $2 j_i$, where $0 \leq j_i \leq n$, and coloring all of the other edges by $n$. See Figure \ref{example2} for an example.
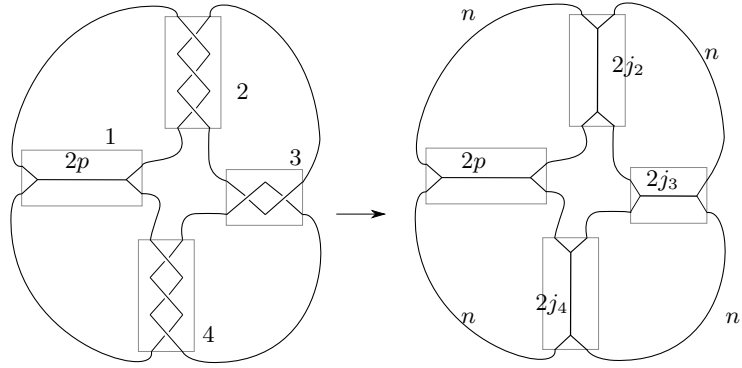
\begin{figure}[H]
	\centering
	{\includegraphics[scale=0.1]{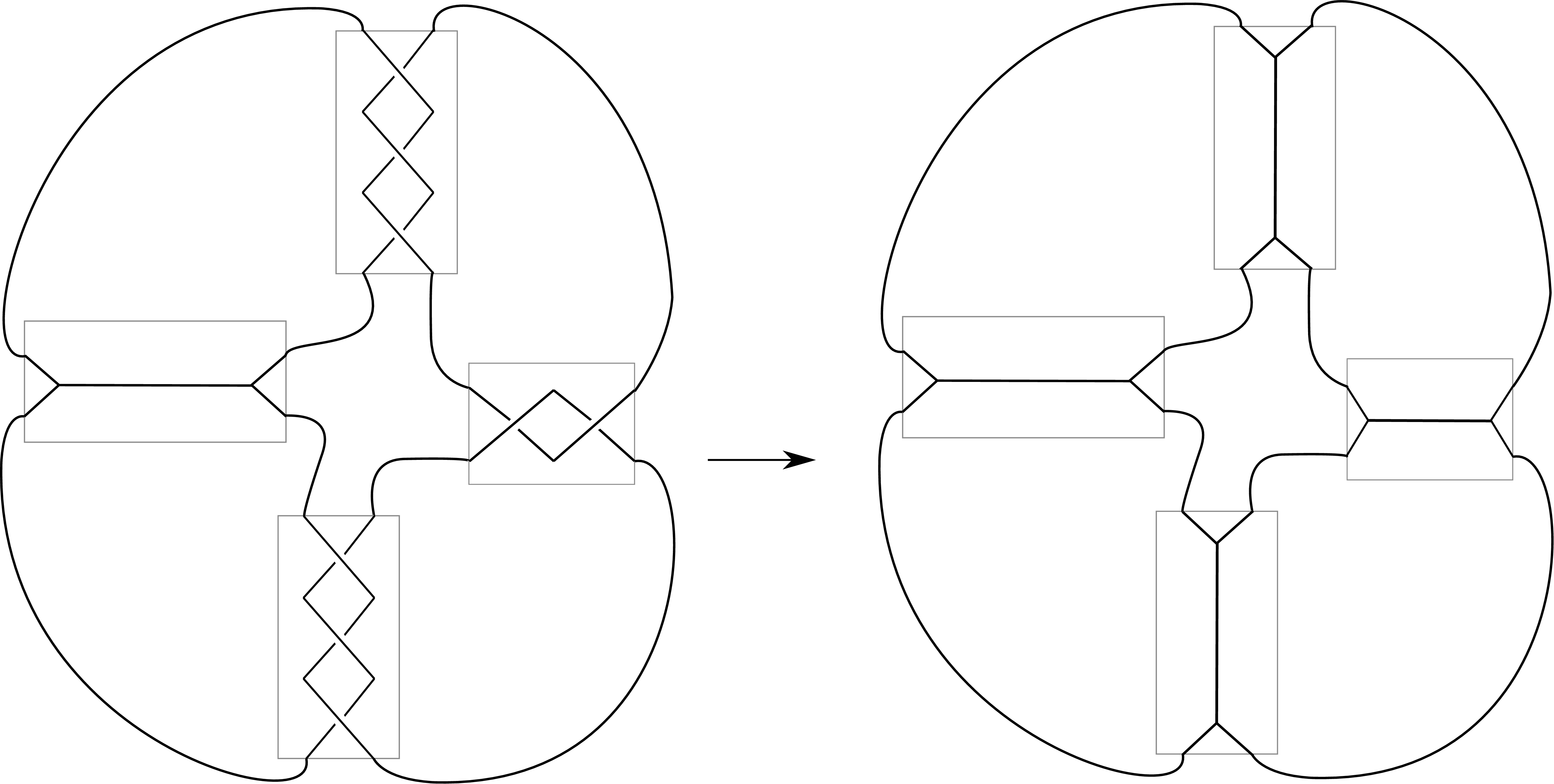}
		\put(-235,+83){\footnotesize{$1$}}
		\put(-185,+100){\footnotesize{$2$}}
		\put(-165,75){\footnotesize{$3$}}
		\put(-198,8){\footnotesize{$4$}}
		\put(-43,+110){\footnotesize{$2j_2$}}
		\put(-30,67){\footnotesize{$2j_3$}}
		\put(-72,20){\footnotesize{$2j_4$}}
		\put(-250,74){\footnotesize{$2p$}}
		\put(-100,74){\footnotesize{$2p$}}
		\put(-100,130){\footnotesize{$n$}}
		\put(-100,15){\footnotesize{$n$}}
		\put(0,15){\footnotesize{$n$}}
		\put(-8,115){\footnotesize{$n$}}
		\caption{On the left the skein element $\Upsilon(n,p)$ obtained from the link $L=L(4,3,2,3)$ shown in Figure \ref{example}. On right the trivalent graph $\Gamma_{n,p,(j_2,j_3,j_4)}$ obtained from $\Upsilon(n,p)$ by replacing the maximal twist regions labeled $2$, $3$ and $4$ by the labels $2j_2, 2j_3$ and $2j_4$ respectively.}
		\label{example2}}
\end{figure}
We want to find the minimum degree of the skein element $\Upsilon(n,p)$. For this purpose we recall the following fact from \cite{Armond}.
\begin{lemma}\cite{Armond}
\label{armond2}
\begin{enumerate}
\item Let $1 \leq j \leq n$, then
\begin{eqnarray*}
m(\mu_{2j}^{n,n})&=&m(\mu_{2(j-1)}^{n,n})-4j.
\end{eqnarray*}
\item 

Let $ 1 \leq j \leq n$, then
\begin{eqnarray*}
m(\frac{\Delta_{2j}}{\theta(n,n,2j)})=m(\frac{\Delta_{2(j-1)}}{\theta(n,n,2(j-1))})-2
\end{eqnarray*}
\item 

Let $1 \leq j_i,p \leq n$ for $2\leq i \leq r$. Then
\begin{eqnarray}
m(\Gamma_{n,p,(j_2,\ldots, j_{i-1},j_i,i_{j+1},\ldots,j_r)})=m(\Gamma_{n,p,(j_2,\ldots, j_{i-1},j_i-1,i_{j+1},\ldots,j_r)}) \pm 2.
\end{eqnarray}

\end{enumerate}
\end{lemma}

The following Lemma studies the minimum degree of the trivalent graph $\Upsilon(n,p)$. 
\begin{lemma}\label{lemma4.4}
\label{technical}
Let $L=L(k_1,\ldots,k_r)$ be a reduced alternating link diagram with $c$ crossings. Let $\Upsilon(n,p)$ be the skein element in $\mathcal{S}(S^2)$ obtained from $L(k_1,\ldots,k_r)$ as illustrated above. Then $m(\Upsilon(n,p))=-n^2(c-k_1)-2 (s_-(L)n-(n-p))$.
\end{lemma}
\begin{proof}
The trivalent graph $\Upsilon(n,p)$ can be written as in equation (\ref{newc}). We claim that the minimum degree of $\Upsilon_{n,p}$ is coming from the term
\begin{equation}
\label{term}
 \prod_{i=2}^{r} (\mu_{2n}^{n,n})^{k_i}\frac{\Delta_{2n}}{\theta(n,n,2n)}\Gamma_{n,p,(n,\ldots,n)}
\end{equation} and the minimal degree of this term does not get canceled by any other terms from the summation (\ref{newc}).
Since $\mu_{2n}^{n,n}=A^{-n^2}$ and $\Delta_{2n}=\theta(n,n,2n)$, we can write \ref{term} as
 $$\prod_{i=2}^{r} (\mu_{2n}^{n,n})^{k_i}\frac{\Delta_{2n}}{\theta(n,n,2n)}\Gamma_{n,p,(n,\ldots,n)}=(A)^{-n^2(c-k_1)}\Gamma_{n,p,(n,\ldots,n)}.$$
We want to find the minimum degree of $\Gamma_{n,p,(n,\ldots,n)}$. We do this by comparing this element to $\Gamma_{n,n,(n,\ldots,n)}$. The state $\overline{\Gamma_{n,n,(n,\ldots,n)}}$ is equivalent to the all-negative state smoothing of the link $L^{n}$. Since $L$ is alternating then $L^n$ is minus-adequate and hence $\Gamma_{n,n,(n,\ldots,n)}$ is an adequate skein element. The skein element $\overline{\Gamma_{n,p,(n,\ldots,n)}}$ is obtained from $\overline{\Gamma_{n,n,(n,\ldots,n)}}$ by merging exactly $n-p$ circles as illustrated in Figure \ref{complicated} $(a)$ and $(b)$. Hence the number of connected components of $\overline{\Gamma_{n,p,(n,\ldots,n)}}$ is $n-p$ fewer than that number of $\overline{\Gamma_{n,n,(n,\ldots,n)}}$. In other words, the number of connected components of the state $\overline{\Gamma_{n,p,(n,\ldots,n)}}$ is $ns_-(L)-(n-p)$. Figure \ref{complicated} (b) shows the  local difference between the skein elements $\Gamma_{n,p,(n,\ldots,n)}$ and $\Gamma_{n,n,(n,\ldots,n)}$. All circles in $ \overline{\Gamma_{n,p,(n,\ldots,n)}}$ in the region outside the illustrated region in Figure \ref{complicated} (b) pass at most once through the regions of the idempotents since the outside of this skein element is identical to the outside of the adequate skein element $\Gamma_{n,n,(n,\ldots,n)}$. On the other hand, each one of the $n-p$ circles inside the region illustrated in Figure \ref{complicated} (b) passes at most once through the region of the idempontents. The same holds for the circles labeled $p$. Hence all circles in $\overline{\Gamma_{n,p,(n,\ldots,n)}}$ pass at most once though the regions of the idemponents and hence $\Gamma_{n,p,(n,\ldots,n)}$ is an adequate skein element. 
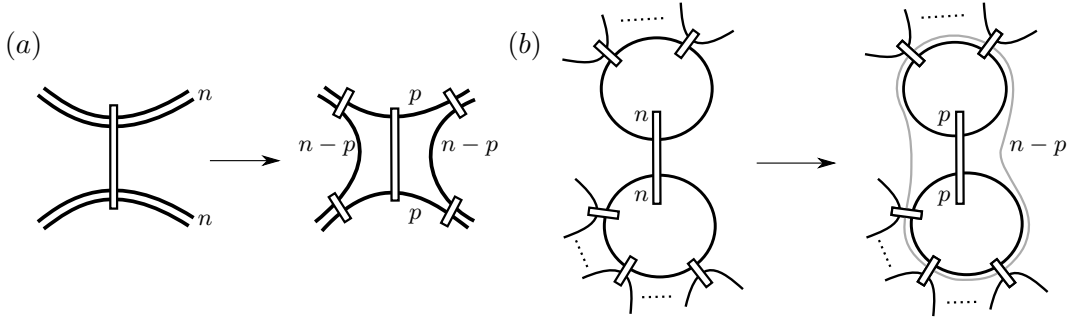
\begin{figure}[H]
	\centering
	{\includegraphics[scale=0.35]{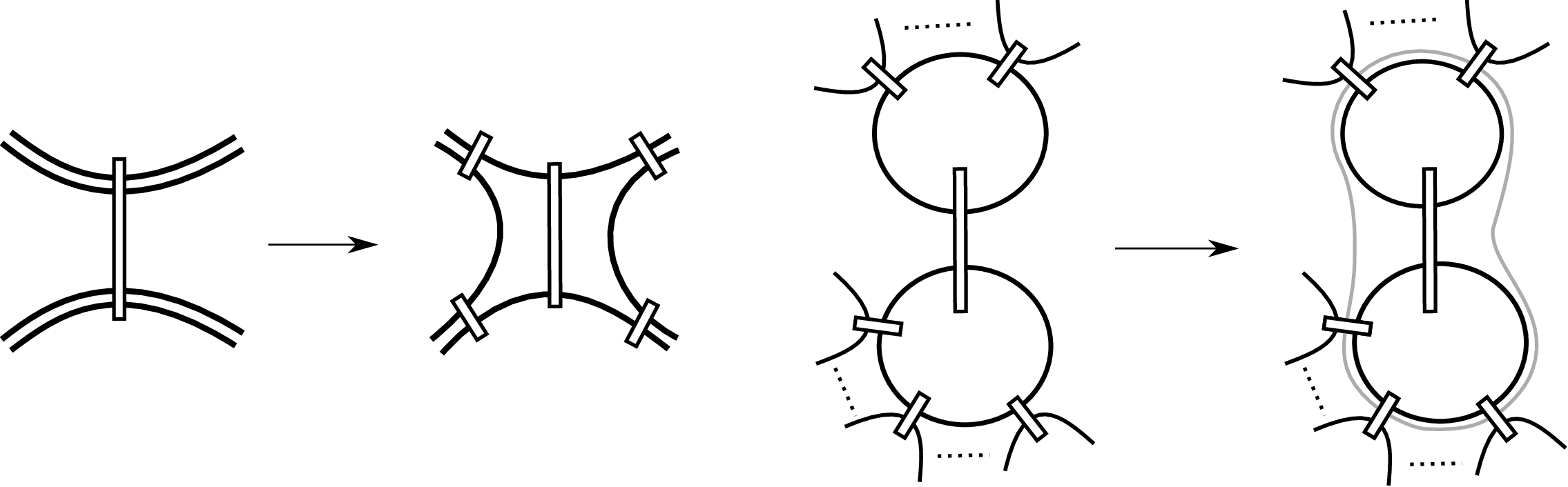}
		\put(-15,+63){\footnotesize{$n-p$}}
		\put(-42,+74){\footnotesize{$p$}}
		\put(-42,+44){\footnotesize{$p$}}
		\put(-157,+74){\footnotesize{$n$}}
		\put(-157,+44){\footnotesize{$n$}}
		\put(-205,+100){$(b)$}
		\put(-322,+82){\footnotesize{$n$}}
		\put(-322,+35){\footnotesize{$n$}}
		\put(-242,+82){\footnotesize{$p$}}
		\put(-242,+37){\footnotesize{$p$}}
		\put(-230,+63){\footnotesize{$n-p$}}
		\put(-284,+63){\footnotesize{$n-p$}}
		\put(-395,+100){$(a)$}
		\caption{(a) The local change that we need to do on the skein element $\Gamma_{n,n,(n,\ldots,n)}$ to obtain the skein element $\Gamma_{n,p,(n,\ldots,n)}$. (b) The skein element $\Gamma_{n,n,(n,\ldots,n)}$ is an adequate skein element and $\Gamma_{n,p,(n,\ldots,n)}$ is obtained from the skein element $\Gamma_{n,n,(n,\ldots,n)}$ by merging $n-p$ circles in $\overline{\Gamma_{n,n,(n,\ldots,n)}}$ to obtain the gray circle. }
		\label{complicated}}
\end{figure}
 Since $\Gamma_{n,p,(n,\ldots,n)}$ is an adequate skein element then by Lemma \ref{ArmondL} we have $m(\Gamma_{n,p,(n,\ldots,n)})=M(\Gamma_{n,p,(n,\ldots,n)})$. However,  $M(\Gamma_{n,p,(n,\ldots,n)})$ is equal to $-2(ns_-(L)-(n-p))$. Thus, $$m((A)^{-n^2(c-k_1)}\Gamma_{n,p,(n,\ldots,n)})= -n^2(c-k_1)-2 (s_-(L)n-(n-p)).$$ It is left to show that none of the terms in the summation (\ref{newc}) cancels the minimum term in $(A)^{-n^2(c-k_1)}\Gamma_{n,p,(n,\ldots,n)}$. This follows immediately from Lemma (\ref{armond2}). The result follows.
\end{proof}
 \begin{remark}
In Lemma \ref{lemma4.4}, we assumed that $r$ is the number of total maximal negative twist regions in the diagram $L$. It should be noted however that the proofs given here work when we choose to label a subset of the maximal twist regions set of the diagram. The reason for choosing $r$ to be maximal is to make the notation of Lemma  less cumbersome.
\end{remark}
\begin{theorem}\label{thm4.6}
\label{queen theorem}
Let $L=L_k$ be a reduced alternating link diagram with a marked maximal negative twist region with $k \geq 1$ crossings. Then,
\begin{equation}
\tilde{J}_n(L_{k})\doteq_{4n(k-1)+4}\tilde{J}_n(L_{k-1}).
\end{equation} 
\end{theorem}	
\begin{proof}
Since
\begin{equation}
\langle S_n(L_k) \rangle \doteq \tilde{J}_n(L_{k}),
\end{equation}
then we will do the computations on $\langle S_n(L_k) \rangle$.
Using equation \ref{greatness} on a single crossing of the marked twist region we obtain : 
 \begin{eqnarray*}
    \langle S_n(L_k) \rangle \doteq \quad  \begin{minipage}[h]{0.15\linewidth}
        \vspace{0pt}
        \scalebox{0.18}{\includegraphics{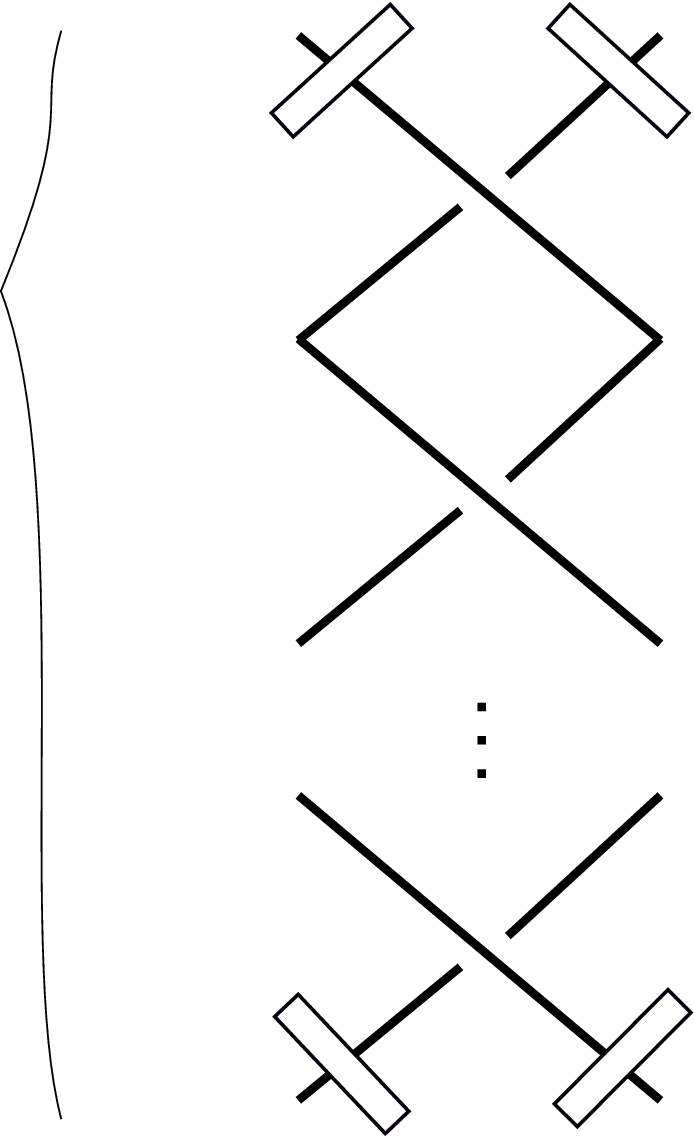}}
          \put(0,+100){\footnotesize{$n$}}
        \put(-44,+100){\footnotesize{$n$}}
         \put(0,+5){\footnotesize{$n$}}
        \put(-44,+5){\footnotesize{$n$}}
        \put(-70,+70){\footnotesize{$k$}}
   \end{minipage}
   &=&\displaystyle\sum\limits_{j=0}^{n-1} C_{n,j}\hspace{1 mm}  
    \begin{minipage}[h]{0.06\linewidth}
        \vspace{0pt}
        \scalebox{0.18}{\includegraphics{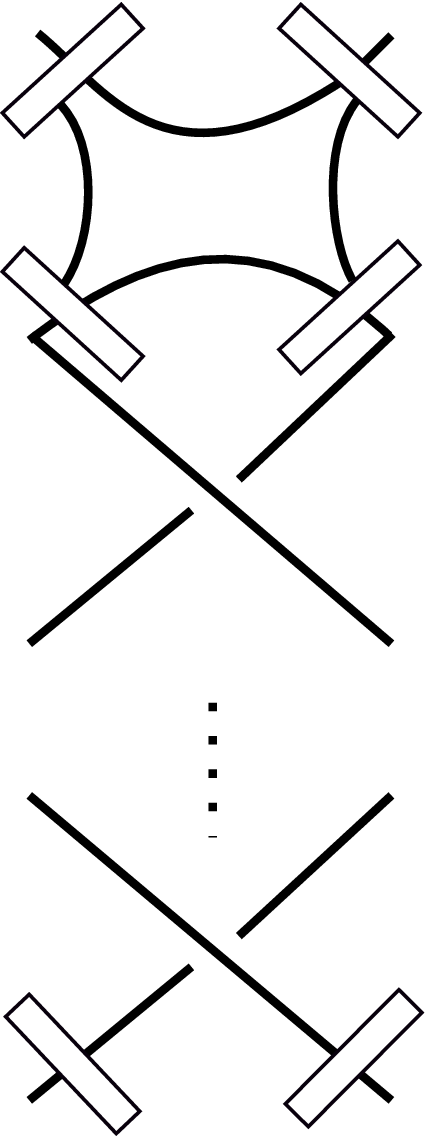}}
             \put(-0,+95){\footnotesize{$n$}}
        	\put(-44,+95){\footnotesize{$n$}}
        	\put(0,+5){\footnotesize{$n$}}
        \put(-44,+5){\footnotesize{$n$}}
        	\put(-34,+79){\tiny{$j$}}
   \end{minipage} \quad +\quad C_{n,n}\hspace{7 mm}  
    \begin{minipage}[h]{0.21\linewidth}
        \vspace{0pt}
        \scalebox{0.18}{\includegraphics{k-crossing}}
          \put(0,+100){\footnotesize{$n$}}
        \put(-44,+100){\footnotesize{$n$}}
         \put(0,+5){\footnotesize{$n$}}
        \put(-44,+5){\footnotesize{$n$}}
        \put(-82,+70){\footnotesize{$k-1$}}
   \end{minipage}\\
   &=&R(n,k)+A^{-n^2} \langle S_n(L_{k-1}) \rangle
\end{eqnarray*}
where
\begin{eqnarray}
R(n,k)=\displaystyle\sum\limits_{j=0}^{n-1} C_{n,j}\hspace{1 mm}  
    \begin{minipage}[h]{0.06\linewidth}
        \vspace{0pt}
        \scalebox{0.18}{\includegraphics{k-crossing_1}}
             \put(-0,+95){\footnotesize{$n$}}
        	\put(-44,+95){\footnotesize{$n$}}
        	\put(0,+5){\footnotesize{$n$}}
        \put(-44,+5){\footnotesize{$n$}}
        	\put(-34,+79){\tiny{$j$}}.
   \end{minipage}.
\end{eqnarray}
 Applying the fusion formula on the skein element that appears in $R(n,k)$, we obtain
\begin{eqnarray*}
R(n,k)=\displaystyle\sum\limits_{j=0}^{n-1}\displaystyle\sum\limits_{p=0}^{j} C_{n,j}\frac{\Delta_{2p}}{\theta(j,j,2p)}\hspace{1 mm}  
    \begin{minipage}[h]{0.06\linewidth}
        \vspace{0pt}
        \scalebox{0.2}{\includegraphics{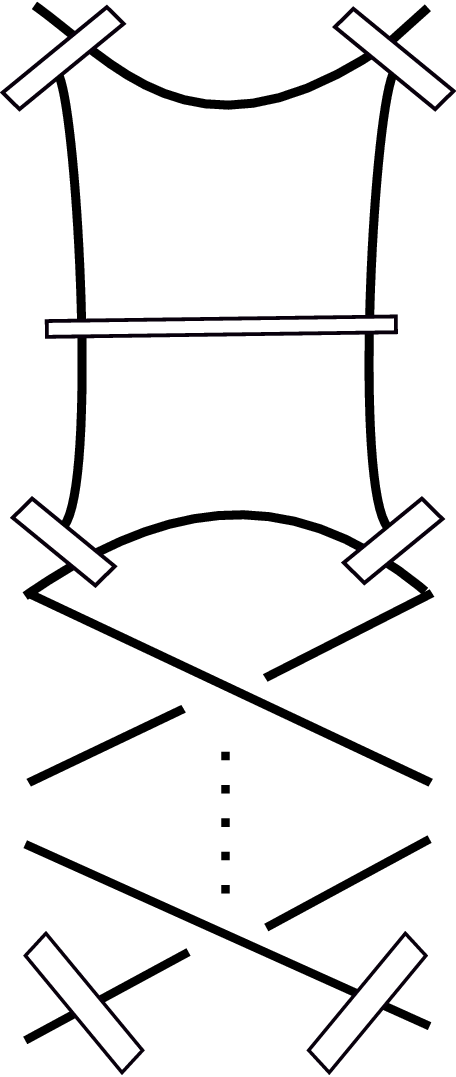}}
             \put(-0,+95){\footnotesize{$n$}}
        	\put(-50,+95){\footnotesize{$n$}}
        	\put(0,+5){\footnotesize{$n$}}
        \put(-50,+5){\footnotesize{$n$}}
        	\put(-31,+98){\tiny{$n-p$}}
        	\put(-31,+47){\tiny{$n-p$}}
        	\put(-44,+76){\tiny{$p$}}.
   \end{minipage}
\end{eqnarray*} 
Using the formula \ref{untwist}, we have 
\begin{eqnarray}
\label{finalR}
R(n,k)=\displaystyle\sum\limits_{j=0}^{n-1}\displaystyle\sum\limits_{p=0}^{j} C_{n,j}\frac{\Delta_{2p}}{\theta(j,j,2p)}(\mu_{2p}^{n,n})^{k-1}\hspace{1 mm}  
    \begin{minipage}[h]{0.06\linewidth}
        \vspace{0pt}
        \scalebox{0.3}{\includegraphics{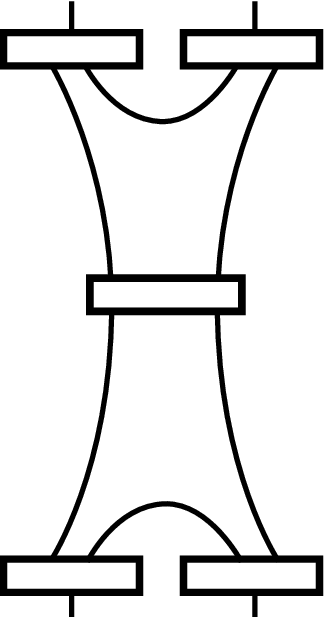}}
             \put(-5,+90){\footnotesize{$n$}}
        	\put(-44,+90){\footnotesize{$n$}}
        	\put(-12,+55){\footnotesize{$p$}}
        	\put(-32,+66){\tiny{$n-p$}}
   \end{minipage}
\end{eqnarray}
where
 \begin{equation*}
 \mu_{2p}^{n,n}=(-1)^{n-p}A^{n^2+2n-2p-2p^2}.
\end{equation*}
Now consider the equation that we derived earlier: 
\begin{equation}
\label{mainproof}
\langle S_n(L_{k}) \rangle=R(n,k)+A^{-n^2} \langle S_n(L_{k-1}) \rangle .
\end{equation}
We want to show that $R(n,k)$ does not contribute to the first $4n(k-1)+4$ terms of $\langle S_n(L_{k}) \rangle$. We show this by proving that the minimum degree of $R(n,k)$ is  $4n(k-1)+4$ higher than the minimum degree of $ A^{-n^2} \langle S_n(L_{k-1}) \rangle$.\\ 

Assume that the diagram $L_k$ has $c$ crossings. This implies that the number of crossings of the diagram $L_{k-1}$ is $c-1$. Moreover, $|s_-(L_{k-1})|=|s_-(L_{k})|$. Since $L_{k-1}$ is alternating, then $m( \langle S_n( L_{k-1}) \rangle )=-(c-1)n^2-2ns_-(L_k)$. Hence,
\begin{equation}
m(A^{-n^2} \langle S_n(L_{k-1})\rangle)= m(\langle S_n(L_k) \rangle). 
\end{equation}

 This implies that $\langle S_n(L_k) \rangle $ comes from the second summand which is to $A^{-n^2} \langle S_n(L_{k-1}) \rangle $.  Now, denote by $\Upsilon(n,p)$ the skein element that appears on the right hand side of equation (\ref{finalR}) and write

 \begin{eqnarray}\label{compli0}
 \label{compli}
    R(n,k)&=&C_{n,n-1}\frac{\Delta_{2(n-1)}}{\theta(n-1,n-1,2n-2)}(\mu_{2n-2}^{n,n})^{k-1} \Upsilon(n,n-1) \nonumber\\ &+&\displaystyle\sum\limits_{p=0}^{n-2} C_{n,n-1}\frac{\Delta_{2p}}{\theta(n-1,n-1,2p)}(\mu_{2p}^{n,n})^{k-1} \Upsilon(n,p) \nonumber \\ &+&\displaystyle\sum\limits_{j=0}^{n-2}\displaystyle\sum\limits_{p=0}^{j} C_{n,j}\frac{\Delta_{2p}}{\theta(j,j,2p)}(\mu_{2p}^{n,n})^{k-1} \Upsilon(n,p).
 \end{eqnarray} 

We claim that the minimum degree of $R(n,k)$ is coming from the first summand of (\ref{compli0}). To see this, first note that Lemma \ref{technical} implies
\begin{eqnarray*}
m(\Upsilon(n,p)) = -n^2(r - k)-2( s_-(L_k)n -(n-p)).
 \end{eqnarray*}
Moreover,
\begin{eqnarray*}
 C_{n,j}\frac{\Delta_{2p}}{\theta(j,j,2p)}(\mu_{2p}^{n,n})^{k-1}&=&(-1)^{-j - n + k n + 2 p - k p} A^{
2 j + 2 j^2 - 2 n - 4 j n + 2 k n + k n^2 - 2 k p + 2 p^2 - 2 k p^2}\\ &\times& \frac{ 
 (A^4; A^4)_{ j} (A^4; A^4)_{n} (A^4; A^4)_{1 + 2 p}}{
(A^4; A^4)_{-j + n} (A^4; A^4)_{j - p} (A^4; A^4)_{1 + j + p} (A^4; A^4)^2_{ p}}.
\end{eqnarray*}

\noindent
Hence,

\begin{eqnarray}
\label{vcomplicated}
m( C_{n,j}\frac{\Delta_{2p}}{\theta(j,j,2p)}(\mu_{2p}^{n,n})^{k-1} \Upsilon(n,p))&=&2 j + 2 j^2 - 4 j n + 2 k n + 2 k n^2 - 2 p - 2 k p + 2 p^2\nonumber\\ &-& 2 k p^2 - n^2 r - 2 n s_-(L_k).
\end{eqnarray}
\noindent
In particular,

\begin{eqnarray*}
m( C_{n,n-1}\frac{\Delta_{2(n-1)}}{\theta(n-1,n-1,2n-2)}(\mu_{2n-2}^{n,n})^{k-1} \Upsilon (n,n-1) )=- 2 n s_-(L_k) - n^2 r +  4n ( k-1)+4.
\end{eqnarray*}
Moreover, from (\ref{vcomplicated}) we see that
\begin{eqnarray*}
m( C_{n,n-1}\frac{\Delta_{2(n-1)}}{\theta(n-1,n-1,2n-2)}(\mu_{2n-2}^{n,n})^{k-1} \Upsilon (n,n-1) ) < m( C_{n,j}\frac{\Delta_{2p}}{\theta(j,j,2p)}(\mu_{2p}^{n,n})^{k-1} \Upsilon(n,p)), 
\end{eqnarray*}
when $j,p \neq n-1 $.
Thus,
\begin{eqnarray*}
m(R(n,k)) &=& m( C_{n,n-1}\frac{\Delta_{2(n-1)}}{\theta(n-1,n-1,2n-2)}(\mu_{2n-2}^{n,n})^{k-1} \Upsilon (n,n-1) )\\&=&- 2 n s_-(L_k) - n^2 r +  4n ( k-1)+4.
\end{eqnarray*}
 Since the minimum degree of the term $R(n,k)$ is higher than the minimum degree of $A^{-n^2} \langle  S_n ( L_{k-1} )\rangle $ by at least $4n( k-1)+4$ then by equation (\ref{mainproof}) we are done.
\end{proof}

%In {CK1} it was shown that if many full twists on a certain number of strands are added then the coefficients of the colored Jones polynomial break up into a certain blocks which are independent of the number of full twists and separated by blocks of zeros whose length increase as we change the number of full-twists. 

 The following three Corollaries immediately follow from Theorem \ref{thm4.6} and they are analogous to Corollary \ref{1}, Theorem \ref{2} and Corollary \ref{3} respectively.

\begin{corollary}
Let $L=L_k$ be a reduced alternating link diagram with a marked maximal negative twist region with $k \geq 1$ crossings. Then the sequence $\{\tilde{J}_n(L_{k+i}) \}_{i \in \mathbb{N}}$ has a well-defined tail.
\end{corollary}

\begin{corollary}
\label{nice}
Let $L=L(k_1,\ldots,k_r)$ be a reduced alternating link diagram. Then
\begin{equation}
\langle S_n( L(k_1,\ldots, k_r))\rangle \doteq_{4n(k-1)+4} \langle S_n( L(k_1-1,\ldots, k_r-1) \rangle
\end{equation}
where $k=\min_{1 \leq i \leq r }(k_i)$.
\end{corollary}
\begin{corollary}\label{Cor4.9}
Let $L=L(k_1,\ldots,k_r)$ be an alternating link diagram. Then the sequence $\{\tilde{J}_n(L(k_1+i,\ldots,k_r+i)) \}_{i \in \mathbb{N}}$ has a well-defined tail.
\end{corollary}
We will denote by $T_{n,L(k_1, \ldots,k_r)}$ the tail in the Corollary \ref{Cor4.9} associated with the alternating link diagram $L(k_1, \dots,k_r)$ and the Jones polynomial.
 
Finally, the following Corollary follows immediately from \ref{nice} and \ref{queen theorem}. 

\begin{corollary}
Let $L=L(k_1,\ldots,k_r)$ be an alternating link diagram. Then for every $n \geq 2$ the sequence $\{J_n(L(k_1+i,\ldots,k_r+i)) \}_{i \in \mathbb{N}}$ stabilizes with a rate $(n-1)k+i+1$ where $k=\min_{1 \leq j \leq r }(k_j)$.
\end{corollary} 
 
\section{Connection with the Tail of the Colored Jones Polynomial}
\label{sec5}
The tail of the unreduced colored Jones polynomial of an alternating link $L$ is a $q$-series $T_L(q)$ that satisfies : \begin{equation}
T_L(q)\doteq \tilde{J}_n(L).
\end{equation}
See \cite{Armond,DL,GL} for more details. This follows from the fact that for every $n\geq 2$ one has \cite{Armond} :
\begin{equation}
\label{ColoredJonesTail}
\tilde{J}_{n}(L) \doteq_{4n} \tilde{J}_{n-1}(L).
\end{equation} It was proven in \cite{CodyOliver} that the tail of an alternating link with a reduced alternating link diagram $L$ depends only the reduced minus-graph of $\mathbb{G}_-^{\prime}(L)$. More precisely we state the following theorem.
\begin{theorem} \cite{Armond}
\label{cody thm2}
Let $L_1$ and $L_2$ be two alternating links with alternating diagrams $D_1$ and $D_2$. If the graph $\mathbb{G}_-^{\prime}(D_1)$ coincides with $\mathbb{G}_-^{\prime}(D_2)$, then $T_{L_1}=T_{L_2}$. 
\end{theorem}
Given a link diagram $L=L(k_1,\ldots,k_r)$, Theorem \ref{cody thm2} implies that adding  negative twists to the $r$ labeled twists in $L$ does not change the tail $T_L(q)$. More precisely we can restate theorem \ref{cody thm2} in terms of our notation :

\begin{theorem}
\label{codythm}
Let $L=L(k_1,\ldots,k_r)$ be a reduced alternating link diagram. Then for every $n\geq 2$ we have  
\begin{equation}
\langle S_n({L(k_1,\ldots, k_r)}) \rangle \doteq_{4n} \langle S_{n-1}({L(k_1+b_1,\ldots, k_r + b_r)}\rangle
\end{equation}
where $b_i \in \mathbb{Z}$ such that $ k_i+ b_i\geq 1 $ for $1\leq i \leq r$.
\end{theorem} 
The rate of stabilization $4n$ is maximal for the sequence $\{\tilde{J}_n(L) \}_{n\in \mathbb{N}}$ where $L$ is an alternating link. This can be seen by considering the coefficient of the colored Jones polynomial of figure-eight knot. Theorem \ref{codythm} also implies the following:
\begin{theorem}
\label{cody_us}
Let $L=L(k_1, \ldots, k_r) $ be a reduced alternating diagram. 
Then the sequence 
$\{ \tilde{J}_{n+i} (L( k_1 +i, \ldots, k_r +i) ) \}_{i\in {\mathbb N}} $ 
has a well-defined tail.
\end{theorem}
\section{Open Questions}\label{sec6}
The tail of the colored Jones polynomial satisfy certain product structures \cite{CodyOliver,Hajij2}. Furthermore, it has found multiple connections with number theory \cite{CodyOliver,GL,Hajij1,Hajij2,Hajij3}. These properties and connections are yet to be addressed for the $q$-series that we introduced here. In \cite{Roz} a categorified version of the tail of the colored Jones polynomial was given. The result in \cite{Roz} is basically a catergofication of Theorem \ref{codythm}. It is an interesting question whether there is a similar categorification for the Theorems \ref{2} and \ref{Cor4.9}. 

Our calculations show that there are higher levels of stability for the coefficients of the colored Jones polynomial of alternating link diagrams. In other words, if we subtract the stabilized tail from the shifted colored Jones polynomials of alternating diagram we obtain another sequence of $q$-series whose coefficients stabilize in the sense of Theorem \ref{nice}. The process of subtracting can be iterated to obtain higher stabilities. See for example \cite{Katie}. We conjecture that this stability holds for all higher order coefficients for all alternating links. Finally, we conjecture that this stability also occur for other quantum invariants.

\
{\it Acknowledgement:}
The authors are grateful to W.E.~Clark for suggestions. We also thank A.~Champanerkar and I.~Kofman for valuable comments. M. Saito was partially supported by NIH R01GM109459.

\end{document}